\documentclass[12pt,psamsfonts]{amsart}

\usepackage{amssymb,amsfonts,amsmath}
\usepackage[all,arc]{xy}
\usepackage{enumerate}
\usepackage{mathrsfs}
\usepackage{fullpage}
\usepackage{xspace}
\usepackage[margin=1.0in]{geometry}
\usepackage{tcolorbox}
\usepackage{tikz-cd}
\usepackage{color}
\usepackage{aliascnt}
\usepackage[foot]{amsaddr}
\usepackage{hyperref}

\usepackage{enumitem}
\setlist[enumerate]{label=$(\mathrm{\arabic*})$, leftmargin=*}
\setlist[itemize]{leftmargin=*}

\newtheorem{thm}{Theorem}[section]

\newaliascnt{theo}{thm}
\newtheorem{theo}[theo]{Theorem}
\aliascntresetthe{theo}

\newaliascnt{cor}{thm}
\newtheorem{cor}[cor]{Corollary}
\aliascntresetthe{cor}

\newaliascnt{prop}{thm}
\newtheorem{prop}[prop]{Proposition}
\aliascntresetthe{prop}

\newaliascnt{lem}{thm}
\newtheorem{lem}[lem]{Lemma}
\aliascntresetthe{lem}

\newaliascnt{conj}{thm}
\newtheorem{conj}[conj]{Conjecture}
\aliascntresetthe{conj}

\newaliascnt{que}{thm}
\newtheorem{que}[que]{Question}
\aliascntresetthe{que}

\newaliascnt{ass}{thm}

\aliascntresetthe{ass}

\newaliascnt{defnot}{thm}
\newtheorem{defnot}[defnot]{Definition-Notation}
\aliascntresetthe{defnot}

\theoremstyle{remark}
\newaliascnt{rem}{thm}
\newtheorem{rem}[rem]{Remark}
\aliascntresetthe{rem}

\theoremstyle{definition}
\newtheorem{defn}[thm]{Definition}

\newtheorem{exmp}[thm]{Example}

\newtheorem{notn}[thm]{Notation}

\newtheorem{conv}[thm]{Convention}

\newcommand{\Z}{\mathbb{Z}\xspace}
\newcommand{\C}{\mathbb{C}\xspace}
\newcommand{\F}{\mathbb{F}\xspace}
\newcommand{\Q}{\mathbb{Q}\xspace}
\newcommand{\G}{\mathbb{G}\xspace}
\DeclareMathOperator{\Spec}{Spec}
\DeclareMathOperator{\res}{res}

\DeclareMathOperator{\alb}{alb}
\DeclareMathOperator{\img}{Im}

\DeclareMathOperator{\Hom}{Hom}
\DeclareMathOperator{\Pic}{Pic}
\DeclareMathOperator{\Gal}{Gal}

\DeclareMathOperator{\Fil}{Fil}

\DeclareMathOperator{\Alb}{Alb}
\DeclareMathOperator{\Symb}{Symb}

\DeclareMathOperator{\Br}{Br}
\DeclareMathOperator{\inv}{inv}
\makeatletter
\let\c@equation\c@thm
\makeatother
\numberwithin{equation}{section}

\title{Divisibility Results for zero-cycles}

\author[*]{Evangelia Gazaki*} \address[*]{\normalfont Department of Mathematics, University of Virginia, 221 Kerchof Hall, 141 Cabell Dr., Charlottesville, VA, 22904, USA. Email: \texttt{eg4va@virginia.edu}}
\author[**]{Toshiro Hiranouchi**}\address[**]{\normalfont Department of Basic Sciences, Graduate School of Engineering, 
Kyushu Institute of Technology, 
1-1 Sensui-cho, Tobata-ku, Kitakyushu-shi, 
Fukuoka 804-8550 JAPAN. Email: \texttt{hira@mns.kyutech.ac.jp}}
\begin{document}

\begin{abstract}
Let $X$ be a product of smooth projective curves over a finite unramified extension $k$ of $\Q_p$. Suppose that the Albanese variety of $X$ has good reduction and that $X$ has a $k$-rational point. We propose the following conjecture. The kernel of the Albanese map $CH_0(X)^0\rightarrow\Alb_X(k)$ is $p$-divisible. When $p$ is an odd prime, we prove this conjecture for a large family of products of elliptic curves and certain principal homogeneous spaces of abelian varieties. Using this, we provide some evidence for a local-to-global conjecture for zero-cycles of Colliot-Th\'{e}l\`{e}ne  and Sansuc (\cite{Colliot-Thelene/Sansuc1981}), and Kato and Saito (\cite{Kato/Saito1986}). 

\end{abstract}

\maketitle

\section{Introduction} Let $X$ be a smooth, projective, and geometrically connected variety over a field $k$. We consider the group $CH_0(X)$ of zero cycles on $X$ modulo rational equivalence. This group is a direct generalization of the Picard group $\Pic(C)$ of a  curve $C$, and as such it inherits many of its properties. Namely, there is a degree map, $CH_0(X)\xrightarrow{\deg}\Z,$ whose kernel will be denoted by $F^1(X)$.
  Moreover, there is a generalization of the Abel-Jacobi map, 
\[F^1(X)\xrightarrow{\alb_X}\Alb_X(k),\] called the \textit{Albanese map} of $X$, where $\Alb_X$ is the dual abelian variety to the Picard variety of $X$. When $X$ has a $k$-rational point, the degree  map is surjective. 
When $k$ is algebraically closed it follows by Roitman's theorem that the Albanese map is also surjective (\cite{Roitman1980}), but this is not always true over arbitrary fields, except in some special cases. Some examples when surjectivity holds include $K3$ surfaces and geometrically rationally connected varieties (in these cases surjectivity holds trivially since $\Alb_X=0$), and products of curves all having a $k$-rational point (\cite{Kahn92b, Raskind/Spiess2000}). 
 Coming to the question of injectivity, unlike the case of curves when the map $\alb_X$ is always injective, in higher dimensions the situation is rather chaotic and the  map $\alb_X$ has often a very significant kernel, which we will denote by $F^2(X)$. Mumford (\cite{Mumford1968}) was the first to find examples of surfaces over $\C$ with enormous $F^2$, in particular not finitely generated. The key feature of these examples was the positive geometric genus, $p_g(X)>0$. 

When $k$ is a finite extension over its prime field, the expectations for the structure of $F^2(X)$ are on the other extreme, predicting that $F^2(X)$ is rather small. When $k$ is a finite field, $F^2(X)$ is indeed finite and its structure can be understood by geometric class field theory (\cite{Kato/Saito1983}, \cite{Kerz/Saito2016}). 
When $k$ is a number field, that is, a finite extension of $\Q$, we have fascinating conjectures due to Beilinson and Bloch. Namely, Beilinson predicts (\cite{Beilinson1984}) that  $F^2(X\otimes_k\overline{\Q})=0$, which would imply that $F^2(X)$ is a torsion group,  while Bloch (\cite{Bloch1984}) expects that the group $CH_0(X)$ is a finitely generated abelian group. The two conjectures combined suggest that $F^2(X)$ should be  finite.
Apart from curves for which the above conjectures follow by the Mordell-Weil theorem, there is some evidence for surfaces with $p_g(X)=0$ by the work of Colliot-Th\'{e}l\`{e}ne and Raskind (\cite{Colliot-TheleneRaskind1991}) and Salberger (\cite{Salberger}), and for the self-product $E\times E$ of an elliptic curve $E$ over $\Q$ by the work of Langer and Saito (\cite{LangerSaito}), and Langer and Raskind (\cite{LangerRaskind}). Namely, for all these classes of surfaces it has been shown that the torsion subgroup of $F^2(X)$ is finite. 


The intermediate case of a $p$-adic field $k$ is rather interesting, as it features similarities with both $\C$ and $\Q$. In this case the group $F^2(X)$ is conjectured by Raskind and Spiess to have the following structure. 

\begin{conj} [{Raskind, Spiess, \cite[Conjecture~3.5.4]{Raskind/Spiess2000}}]\label{local}  Let $X$ be a smooth projective variety over a finite extension of the $p$-adic field  $\Q_p$. The Albanese kernel $F^2(X)$ is the direct sum of a divisible group and a finite group. 
\end{conj} This conjecture was inspired by earlier considerations of Colliot-Th\'{e}l\`{e}ne, \cite[Conjecture~1.4 (d, e, f)]{Colliot-Thelene1993}. 
A celebrated result in this direction is due to Saito and Sato (\cite{Saito/Sato2010}), who proved a weaker form of \autoref{local}, namely that the group $F^1(X)$  is the direct sum of a finite group and a group divisible by any integer $m$ coprime to $p$.
 The full conjecture has been verified in very limited cases including 
rationally connected varieties with good reduction  \cite[Theorem 5]{Kollar/Szabo2003} (those in fact satisfy $F^2(X)=0$) and certain products of curves, \cite[Theorem 1.1]{Raskind/Spiess2000}, \cite[Theorem 1.2]{Gazaki/Leal2018}. Moreover, for rationally connected varieties with semistable reduction and abelian varieties with good ordinary reduction it has been established that the group $F^2(X)$ is the direct sum of a divisible group and a torsion group (cf.\ \cite[Corollary 9]{Kollar/Szabo2003},  and \cite[Theorem 1.1]{Gazaki2019} respectively). We also refer to \cite{Parimala/Suresh1995} for a list of results on $CH_0$ for quadric fibrations over various types of fields, including number fields. 
 
 When $X$ has good reduction, it follows by \cite[Theorem 0.3, Corollary 0.10]{Saito/Sato2010} and \cite[Theorem 1]{Kato/Saito1983} that the group $F^1(X)$ is $l$-divisible for every prime $l\neq p$. For the classes of varieties for which \autoref{local} has been verified, this is also true for the group $F^2(X)$. It is natural therefore to ask if the same holds for the ``$p$''-part, a question which only makes sense for the group $F^2(X)$. 
 \begin{que}\label{introque} Suppose that the variety $X$ has good reduction and that  \autoref{local} holds for $X$. Is the group $F^2(X)$ $p$-divisible? 
 \end{que}
 The answer is no in general. In fact all the known results (\cite{Raskind/Spiess2000, Yamazaki2005, Hiranouchi2014, Gazaki/Leal2018})  indicate that the group $F^2(X)/p$ is nontrivial when $k$ is \textit{ramified enough}. 
 The purpose of this article is to investigate what happens when $k$ is unramified over $\Q_p$. We expect that in this case \autoref{introque} should have an affirmative answer, at least for certain classes of varieties,  including some cases of bad reduction.  This expectation is strongly motivated by certain local-to-global expectations for zero-cycles, which will be discussed in more detail in \autoref{localtoglobalintro}. 
We suggest the following conjecture. 
\begin{conj}\label{unramifieddivisibility} Suppose that $k$ is a finite unramified extension of $\Q_p$. Let $X=C_1\times\cdots\times C_r$ be a product  of smooth projective curves over $k$ such that for $i=1,\ldots,r$, $C_i(k)\neq\emptyset$.  Suppose we are in one of the following two situations.
\begin{enumerate}[label=$(\alph*)$]
\item The Jacobian variety $J_i$ of $C_i$ has good reduction, for  $i=1,\ldots,r$. 
\item The Jacobian variety $J_i$ of $C_i$ has split multiplicative reduction, for $i=1,\ldots,r$, that is, $C_i$ is a Mumford curve over $k$. 
\end{enumerate}  
 Then, the kernel of the Albanese map $F^2(X)$ is $p$-divisible. 
 \end{conj}
 
Our first significant evidence  for the situation ($a$) of the above conjecture is the following theorem, which constitutes the main result of this article. 
\begin{theo}[cf.\ \autoref{main1}]\label{maintheointro}  Let $X=E_1\times E_2$ be a product of two elliptic curves  over a finite unramified extension $k$ of $\Q_p$ with good reduction, where $p$ is an odd prime. Suppose that one of the curves has good ordinary reduction. Then, the Albanese kernel $F^2(X)$ is $p$-divisible. 
\end{theo} 
 We note that \autoref{local} has already been established for such products by joint work of the first author with Isabel Leal (\cite[Theorem 1.2]{Gazaki/Leal2018}).  
Moreover, using easy descent arguments, and related work of Takao Yamazaki (\cite{Yamazaki2005}), we verify \autoref{unramifieddivisibility} in the following additional cases,  including the situation ($b$) of \autoref{unramifieddivisibility}. 
\begin{cor}[cf.\ Corollaries \ref{product}, \ref{extensions}, and \ref{Mumford curves}] Let $k$ be a finite unramified extension of $\Q_p$, with $p$ is odd. Then, \autoref{unramifieddivisibility} is true for each of the following classes of varieties. 
\begin{enumerate}[label=$(\mathrm{\alph*})$]
\item An abelian variety $A$ such that there is an isogeny $A\xrightarrow{\phi}E_1\times\cdots\times E_r$ of degree coprime to $p$, where $E_i$ are elliptic curves over $k$ with good reduction with at most one having good supersingular reduction. 
\item A principal homogeneous space $X$ of an abelian variety $A$, such that  $X\otimes_k L\simeq A\otimes_k L$ for some finite extension $L/k$ of degree coprime to $p$ and  with $A$ as in $(\mathrm{a})$.
\item  A product $X=C_1\times\cdots\times C_r$ of Mumford curves over $k$. 
\end{enumerate}
\end{cor} 

Our current techniques work well only for odd primes $p$. However, we have no significant reason to exclude $p=2$ from \autoref{unramifieddivisibility}. 
When $X=E_1\times E_2$ with both curves having good supersingular reduction we only obtain  a partial result, which we will state in Section~\ref{mainresults}. Before doing that, we would like to discuss some connections of \autoref{unramifieddivisibility} with certain local-to-global expectations.

\subsection{Local-to-Global Approximations for zero-cycles}\label{localtoglobalintro}
For a smooth projective geometrically connected variety $X$ over a  number field $F$, it is customary to consider the diagonal embedding $X(F)\hookrightarrow X(\mathbf{A}_F)$ to the set of adelic points, $X(\mathbf{A}_F):=\prod_{v\in\Omega}X(F_v)$, 
 where $\Omega$ is the set of places in $F$.  When $X(F)\neq\emptyset$, the question that arises is whether $X$ satisfies weak approximation, that is, whether $X(F)$ is dense in $X(\mathbf{A}_F)$. The Brauer group $\Br(X)$ of $X$ is known to often obstruct Weak Approximation (cf.~\cite{Poonen2017}, \cite{Skorobogatov2001}). Namely, it gives rise to an intermediate closed subset, $X(F)\subset X(\mathbf{A}_F)^{\Br(X)}\subset X(\mathbf{A}_F)$, which is often properly contained in $X(\mathbf{A}_F)$. This obstruction is called \textit{Brauer-Manin obstruction to Weak Approximation}.

 Although this obstruction cannot always explain the failure of Weak Approximation for points, its zero-cycles counterparts are conjectured to explain all phenomena. We are particularly interested in the following conjecture.
\begin{conj}[{\cite[Section 4]{Colliot-Thelene/Sansuc1981}, \cite[Section 7]{Kato/Saito1986}, see also \cite[Conjecture~1.5 (c)]{Colliot-Thelene1993} and \cite[Conjecture~($E_0$)]{Wittenberg2012}}]
\label{locatoglobalconj} \label{ltgconj}
Let $X$ be a smooth projective geometrically connected variety over a  number field $F$. 
 The following complex is exact, 
\[\varprojlim_{n}  F^1(X)/n\xrightarrow{\Delta}
\varprojlim_{n}F^1_{\mathbf{A}}(X)/n\rightarrow\Hom(\Br(X)/\Br(F),\Q/\Z).\] 
\end{conj}  
Here, the adelic Chow group $F^1_{\mathbf{A}}(X)$ is essentially $\prod_{v\in\Omega_f}F^1(X\otimes_F F_v)$ with a small contribution from the infinite real places, 
where $\Omega_f$ is the set of all finite places in $F$.  For a precise definition see Section~\ref{ltgsection}. 

\autoref{ltgconj} was originally suggested by Colliot-Th\'{e}l\`{e}ne and Sansuc (\cite{Colliot-Thelene/Sansuc1981}) for geometrically rational varieties built upon some evidence. It was later extended to general varieties by Kato and Saito (\cite{Kato/Saito1983}). However, to this day the only evidence we have for this  conjecture is for several classes of rationally connected varieties, starting with the work of Colliot-Th\'{e}l\`{e}ne, Sansuc and Swinnerton-Dyer on Ch\^{a}telet surfaces (\cite{CT/Sansuc/SwinnertonDyerI, CT/Sansuc/SwinnertonDyerII}), and continued by multiple authors (cf.\ \cite{Wittenberg2018} for a great survey article). There is some recent partial evidence for $K3$ surfaces by work of Ieronymou (\cite{Ieronymou2018}), which is the only known result for varieties with positive geometric genus.

We are interested to see whether the above conjecture has any chance of being true for a product $X=C_1\times C_2$ of two curves over $F$ having an $F$-rational point. In what follows we do a quantitative analysis of this problem. In this case the Albanese variety $\Alb_X$ is just the product $J_1\times J_2$ of the Jacobian varieties of $C_1, C_2$, and it follows by a result of Raskind and Spiess (\cite[Corollary 2.4.1]{Raskind/Spiess2000}) that we have a decomposition \[CH_0(X)\simeq\Z\oplus \Alb_X(F)\oplus F^2(X)\simeq\Z\oplus J_1(F)\oplus J_2(F)\oplus F^2(X).\] If we assume the finiteness of the Tate-Shafarevich group of $J_1\times J_2$, then verifying \autoref{locatoglobalconj} is reduced to proving that the following complex is exact (cf.\ \autoref{complexreduction}),
\begin{equation}\label{reduce}\varprojlim_{n}  F^2(X)/n\xrightarrow{\Delta}
\varprojlim_{n}F^2_{\mathbf{A}}(X)/n\rightarrow \Hom(\Br(X)/\Br_1(X),\Q/\Z),\end{equation} 
where $\Br_1(X):=\ker(\Br(X)\rightarrow\Br(X\otimes_F\overline{F}))$ is the \textit{algebraic Brauer group}, and the quotient $\Br(X)/\Br_1(X)$ is the \textit{transcendental Brauer group} of $X$. By a result of Skorobogatov and Zarhin (\cite{Skorobogatov/Zharin2014}) the quotient $\Br(X)/\Br_1(X)$ is finite  for such a product of curves $X$  defined over a number field. 
At the same time, the Beilinson-Bloch conjectures predict that the group $F^2(X)$ is finite, and should therefore coincide with $\varprojlim_{n}  F^2(X)/n$. As a conclusion, for \autoref{locatoglobalconj} to be compatible with the global expectations, the adelic Albanese kernel $\varprojlim_{n}F^2_{\mathbf{A}}(X)/n$ must also be finite. \autoref{unramifieddivisibility} when combined with \autoref{local} and \cite[Theorem 3.5]{Raskind/Spiess2000} imply this finiteness, suggesting that 
 the group $\varprojlim_{n}F^2_{\mathbf{A}}(X)/l^n$ vanishes, for all rational primes $l$ lying below unramified places of $F$ of good reduction. \autoref{maintheointro} yields the following corollary, which is precisely of that flavor. 

\begin{cor}[cf.\ \autoref{localtoglobal}] Let $X=E_1\times E_2$ be the product of two elliptic curves over a number field $F$.  Assume that for $i=1,2$ the elliptic curve $E_i$ has potentially good reduction at all  finite places of $F$.  There is an infinite set $T$ of  rational primes such that $\prod_{v\in\Omega}\prod_{l\in T}\varprojlim_{n}F^2(X\otimes_F F_v)/l^n=0$. In particular, the result holds when for $i=1,2$ the elliptic  curve $E_i\otimes_\Q \overline{\Q}$ has complex multiplication by the ring of integers of a quadratic imaginary field $K_i$. 
\end{cor} 
Unfortunately, the complement of  $T$  may be infinite, because in order to use \autoref{maintheointro}, we need to exclude  the rational primes below  
 all places of bad reduction, all ramified places, and all places where both curves have good supersingular reduction, and the latter subset is infinite.
In the present article we can only treat the case of potentially good reduction. We hope that in a future paper we will explore higher ramification cases and primes of bad multiplicative reduction, where it is very likely that global zero-cycles need to be constructed. 
 
\begin{rem}
One could suggest extending \autoref{unramifieddivisibility} to $K3$ surfaces, as most of the above analysis carries over to that case. Namely, in this case the groups $F^1(X)$ and $F^2(X)$ coincide, which allows once again to reduce \autoref{ltgconj} to proving exactness of \eqref{reduce}. Additionally, the quotient $\Br(X)/\Br_1(X)$ is finite (\cite[Theorem 1.2]{Skorobogatov/Zharin2008}). Our methods do not provide any information for those at the moment. However, if one could verify \autoref{unramifieddivisibility} for $K3$ surfaces, this would strengthen very significantly the recent result of Ieronymou (\cite[Theorem 1.2]{Ieronymou2018}), which could potentially lead to a full proof of \autoref{ltgconj} for $K3$ surfaces. Another interesting case to consider is surfaces with $p_g(X)=0$, which are not rationally connected. These are known to have finite Albanese kernel (\cite{Colliot-TheleneRaskind1991}, see also \cite[Th\'{e}or\`{e}me 2.2]{Colliot-Thelene1993}). 
\end{rem}

\subsection{Outline of the method and additional results}\label{mainresults} 
The key tool to prove \autoref{maintheointro} is the use of the Somekawa $K$-group  $K(k;E_1,E_2)$ attached to $E_1, E_2$. This group is a quotient of the group $\displaystyle\bigoplus_{L/k\text{ finite}}E_1(L)\otimes E_2(L)$, and it is a generalization of the Milnor $K$-group $K_2^M(k)$ of $k$. For a finite extension $L/k$ and points $a_i\in E_i(L)$, the image of a tensor $a_1\otimes a_2$ inside $K(k;E_1,E_2)$ is denoted as a symbol $\{a_1,a_2\}_{L/k}$. Raskind and Spiess (\cite{Raskind/Spiess2000}) proved an isomorphism, $\rho:F^2(X)\xrightarrow{\simeq}K(k;E_1,E_2)$. As an example, if $(x,y)\in X(k)$ is a $k$-rational point, then $\rho$ sends the zero-cycle $[x,y]-[x,0]-[0,y]+[0,0]\in F^2(X)$ to the symbol $\{x,y\}_{k/k}$. 

Some limited cases of \autoref{maintheointro} were obtained in \cite{Gazaki/Leal2018} for good ordinary reduction, but the arguments were very ad hoc. 
In the current article, we develop a uniform method to prove $p$-divisibility of $K(k;E_1,E_2)$. Our method roughly involves the following main steps. 

\noindent
\textbf{Step 1:} (cf.\ \autoref{main2}) We show that the $K$-group $K(k;E_1,E_2)/p$ is generated by symbols of the form $\{x,y\}_{k/k}$ for $(x,y)\in X(k)$.  

\noindent
\textbf{Step 2:} (cf.\ \autoref{main1}) We prove that all symbols of the form $\{x,y\}_{k/k}$ are $p$-divisible.  

The key to prove both steps is to consider the extension $L=k(\widehat{E}_1[p], \widehat{E}_2[p])$, where $\widehat{E}_i$ is  the formal group of $E_i$, and study the restriction map \[K(k;E_1,E_2)/p\xrightarrow{\res_{L/k}}K(L;E_1,E_2)/p.\] 
We may reduce to the case when the extension  $L/k$ is of degree coprime to $p$, in which case $\res_{L/k}$
 is injective. The advantage of looking at the restriction is that under the reduction assumptions of \autoref{maintheointro}, we have a complete understanding of the group $K(L;E_1,E_2)/p$; namely it is isomorphic to $(\Z/p\Z)^r$ for some $0\leq r\leq 2$. 
 
When both curves have good supersingular reduction, Step 2 still holds. It is much harder to establish Step 1 however. In this case we obtain the following partial result. 
\begin{theo}[cf.\ Theorems \ref{ssing1}, \ref{main3}]\label{main3intro}  Let $k$ be a finite unramified extension of $\Q_p$. Let $X=E\times E$ be the self product of an elliptic curve over $k$ with good supersingular reduction. Let $L=k(E[p])$.
Then all symbols of the form $\{a,b\}_{L/k}$ and $\{a,b\}_{k/k}$ vanish in $K(k;E,E)/p$. 
\end{theo}

In order to give some content to \autoref{main3intro}, we note that in all other cases it has been established by previous work of Isabel Leal and the authors (\cite{Gazaki/Leal2018, Hiranouchi2014}) that the group $K(L;E_1,E_2)/p$ is generated by symbols $\{x,y\}_{L/L}$ defined over $L$. We expect the same to be true in the case of two elliptic curves with good supersingular reduction, and then \autoref{main3intro} would imply $p$-divisibility. However, this case appears to be much harder, and at the moment we do not have a good way to control all finite extensions. 

\subsection{Acknowledgements}  We are very grateful to Professor Jean-Louis Colliot-Th\'{e}l\`{e}ne for very useful suggestions and advice regarding mainly the local-to-global part of the paper. We would also like to heartily thank Professors Kazuya Kato, Shuji Saito and Takao Yamazaki for showing interest in our paper and for useful discussions and suggestions. Moreover, we are truly thankful to the referees, whose suggestions helped improve significantly this article. The first author was partially supported by the NSF grant DMS-2001605. 
	The second author was supported by JSPS KAKENHI Grant Number 20K03536.
	
\subsection{Notation} For a variety $X$ over a field $k$ and an extension $L/k$, we will denote by $X_L:=X\otimes_k L$ the base change to $L$. For an abelian group $A$ and an integer $n\geq 1$, we will denote by $A[n]$ and $A/n$ the $n$-torsion and $n$-cotorsion respectively. For a field $k$ we will denote by $G_k:=\Gal(\overline{k}/k)$ the absolute Galois group of $k$. Moreover, for a $G_k$-module $M$ we will denote by $H^i(k,M)$ the Galois cohomology groups of $M$, for $i\geq 0$.

\smallskip
\section{Background} In this section we review some necessary background. We start with the definitions of Mackey functors and Somekawa $K$-groups. 
Throughout this section, $k$ will be a field with characteristic 0. 

\subsection{Mackey functors} 
Following \cite[(3.2)]{Raskind/Spiess2000}, we introduce Mackey functors and their product. 
	For properties of Mackey functors, see also \cite{Kahn92a},  \cite{Kahn92b}.
\begin{defn}
A \textbf{Mackey functor} $\mathcal{M}$ (over $k$) is a presheaf of abelian groups in the category of \'{e}tale $k$-schemes equipped with  push-forward maps $f_\star:\mathcal{M}(X)\rightarrow\mathcal{M}(Y)$ for finite morphisms $X\xrightarrow{f}Y$, satisfying the following properties.
\begin{enumerate}[label=(\roman*)]
\item $\mathcal{M}(X_1\sqcup X_2)=\mathcal{M}(X_1)\oplus\mathcal{M}(X_2)$, for \'{e}tale $k$-schemes $X_1,X_2$.
\item If  $Y'\xrightarrow{g}Y$ is a finite morphism and $\begin{tikzcd}
	X'\ar{r}{g'}\ar{d}{f'}  & X\ar{d}{f} \\
	Y'\ar{r}{g} & Y
	\end{tikzcd}$ is a Cartesian diagram, then the induced diagram 
	$\begin{tikzcd}
	\mathcal{M}(X')\ar{r}{g'_\star}  & \mathcal{M}(X) & \\
	\mathcal{M}(Y')\ar{r}{g_\star}\ar{u}{f^{'\star}}  & \mathcal{M}(Y) \ar{u}{f^\star} 
	\end{tikzcd}$ commutes. 
\end{enumerate} \end{defn}
Property (i) implies that a Mackey functor $\mathcal{M}$ is fully determined by its value on $\Spec(K)$ where $K$ is a finite field extension of $k$. We will denote by $\mathcal{M}(K):=\mathcal{M}(\Spec K)$. 
\begin{notn} For finite field extensions $k\subset K\subset L$, the   map $j_\star:\mathcal{M}(L)\rightarrow\mathcal{M}(K)$ induced by the projection $\Spec (L)\xrightarrow{j}\Spec(K)$ will be denoted by $N_{L/K}:\mathcal{M}(L)\rightarrow\mathcal{M}(K)$ and will be referred to as \textbf{the norm}. Similarly, the induced pullback map $j^\star:\mathcal{M}(K)\rightarrow\mathcal{M}(L)$ will be denoted by $\res_{L/K}:\mathcal{M}(K)\rightarrow\mathcal{M}(L)$ and will be referred to as \textbf{the restriction}. 
\end{notn}
The category of Mackey functors over $k$ is abelian (\cite[p.~5]{Kahn/Yamazaki2013}, \cite[p.~14]{Raskind/Spiess2000}) with a tensor product  defined by Kahn in \cite{Kahn92a}, whose definition we review below. 
\begin{defn} For Mackey functors $\mathcal{M},\mathcal{N}$, their \textbf{Mackey product} $\mathcal{M}\otimes \mathcal{N}$ is defined as follows. For a finite field extension $k'/k$,
\begin{eqnarray}\label{MFprod}
&&(\mathcal{M}\otimes \mathcal{N})(k')=\left.\left(\bigoplus_{K/k':\text{ finite}}\mathcal{M}(K)\otimes_\Z \mathcal{N}(K)\right)\middle/(\text{\textbf{PF}}),\right.
\end{eqnarray} where (\textbf{PF}) is the subgroup generated by elements of the following form. For a tower of finite extensions $k'\subset K\subset L$, 
\begin{itemize}
\item[] (PF1) $N_{L/K}(x)\otimes y-x\otimes\res_{L/K}(y)\in (\text{\textbf{PF}})$, for elements $x\in\mathcal{M}(L)$, $y\in\mathcal{N}(K)$.
\item[] (PF2) $x\otimes N_{L/K}(y)-\res_{L/K}(x)\otimes y\in(\text{\textbf{PF}})$, for elements $x\in\mathcal{M}(K)$, $y\in\mathcal{N}(L)$.
\end{itemize} 
These relations are referred in the literature as \textbf{projection formula}.

 For $x\in\mathcal{M}(K),y\in\mathcal{N}(K)$ the image of $x\otimes y$ in $(\mathcal{M}\otimes\mathcal{N})(k')$ is traditionally denoted as a symbol $\{x,y\}_{K/k'}$. Moreover, for a finite extension $F/k'$ the norm map $N_{F/k'}$ is given by,
\begin{eqnarray*}
N_{F/k'}:&&(\mathcal{M}\otimes\mathcal{N})(F)\rightarrow (\mathcal{M}\otimes\mathcal{N})(k')\\&&
\{x,y\}_{K/F}\mapsto\{x,y\}_{K/k'}.
\end{eqnarray*} In other words, $N_{F/k'}(\{x,y\}_{K/F})=\{x,y\}_{K/k'}$, for $x\in\mathcal{M}(K)$, $y\in\mathcal{N}(K)$. 
\end{defn} 
\begin{rem} Using the symbolic notation, the projection formula (\textbf{PF}), can be rewritten as,
\begin{eqnarray}\label{projectionformula}\{N_{L/K}(x),y\}_{K/k'}=\{x,\res_{L/K}(y)\}_{L/k'}, & & \{x,N_{L/K}(y)\}_{K/k'}=\{\res_{L/K}(x),y\}_{L/k'}.\end{eqnarray} 
\end{rem}

\begin{exmp}\label{MFexs} 
\begin{enumerate}[leftmargin=0pt]
\item[] (1)  Let $G$ be a commutative algebraic group over $k$. Then $G$ induces a Mackey functor by defining $G(K):=G(\Spec K)$ for $K/k$ finite. 

\item[] (2) Let $\mathcal{M}$ be a Mackey functor and $n\in\mathbb{N}$ be a positive integer. We define a Mackey functor $\mathcal{M}/n$ as follows: $(\mathcal{M}/n)(K):=\mathcal{M}(K)/n$. 

\item[] (3) 
For every integer $n\geq 1$ and any finite extension $K/k$, 
in an unpublished work due to Kahn, we have an isomorphism
\[\left(\frac{\G_m\otimes\G_m}{n}\right)(K)\simeq \frac{K_2^M(K)}{n},\] where $K_2^M(K)$ is the Milnor $K$-group of $K$ (cf.\ \cite[Remark 4.2.5 (b)]{Raskind/Spiess2000}). 
When $k$ is a finite extension of the $p$-adic field $\Q_p$, 
this follows also from \cite[Lemma 4.2.1]{Raskind/Spiess2000}. 
\end{enumerate}
\end{exmp}

\subsection*{The restriction map} Suppose $k\subset K\subset L$ is a tower of finite extensions of $k$. In the following sections we are going to use extensively the restriction map \[\res_{L/K}:(\mathcal{M}\otimes\mathcal{N})(K)\rightarrow(\mathcal{M}\otimes\mathcal{N})(L).\] We review its definition here. Let $F/K$ be a finite extension. There is an isomorphism of $L$-algebras,
$F\otimes_K L\simeq\prod_{i=1}^nA_i$, where for each $i\in\{1,\ldots,n\}$, 
$A_i$ is an Artin local ring of length $e_i$ over $L$ with residue field $L_i$.
 Let $x\in\mathcal{M}(F)$, $y\in\mathcal{N}(F)$. Then, we define 
\[\res_{L/K}(\{x,y\}_{F/K})=\sum_{i=1}^ne_i\{\res_{L_i/F}(x),\res_{L_i/F}(y)\}_{L_i/L}.\] 

The following lemma gives a more concrete description of $\res_{L/k}$ in some special cases. 
\begin{lem}\label{restriction}
\begin{enumerate}
\item Suppose $L/k$ is a finite extension and $x\in\mathcal{M}(k)$, $y\in\mathcal{N}(k)$. Then $\res_{L/k}(\{x,y\}_{k/k})=\{x,y\}_{L/L}$. 
\item Suppose $L/k$ is a finite Galois extension and $x\in\mathcal{M}(L)$, $y\in\mathcal{N}(L)$. Let $G=\Gal(L/k)$. Then,
$\res_{L/k}(\{x,y\}_{L/k})=\res_{L/k}(N_{L/k}(\{x,y\}_{L/L}))=\sum_{g\in G}g\{x,y\}_{L/L}$. 
\end{enumerate}
\end{lem}
\begin{proof}
The assertion (1) follows immediately, since we have an isomorphism $k\otimes_k L\simeq L$. Hence, in this case we have $n=1$, $L_i=L$ and $e_i=1$.

 To prove (2),  note that since $L/k$ is Galois, 
it is the splitting field of some polynomial $f(x) \in k[x]$ 
so that $L \simeq k[x] /\langle f(x) \rangle$. 
By taking a root $\alpha\in L$ of $f(x)$, we have 
 \[L\otimes_k L=\frac{k[x]}{\langle f(x)\rangle}\otimes_k L=\frac{L[x]}{\langle f(x)\rangle}\simeq\prod_{g \in G}\frac{L[x]}{\langle x-g(\alpha) \rangle}\simeq\prod_{i=1}^n L_i,\] where for each $i\in\{1,\ldots,n\}$ there is an isomorphism $g:L_i\xrightarrow{\simeq}L$ given by some $g\in G$. 
\end{proof}
\begin{rem}\label{normrestr} Let $\mathcal{M}$ be a Mackey functor over $k$.  
Suppose that we have an equality $N_{L/K}\circ\res_{L/K}=[L:K]$, 
for every finite extension $L/K$, e.g.\ $\mathcal{M}$ is given by a commutative algebraic group (cf.\ Example~\ref{MFexs} (1)). 
The equality implies that for every prime number $p$ coprime to $[L:K]$, the restriction map $(\mathcal{M}/p)(K)\xrightarrow{\res_{L/K}}(\mathcal{M}/p)(L)$ is injective, while the norm $(\mathcal{M}/p)(L)\xrightarrow{N_{L/K}}(\mathcal{M}/p)(K)$ is surjective. 
\end{rem}
\begin{rem}\label{vanishingtrick} (\textbf{The vanishing trick}) In the following sections we are going to prove the vanishing of certain Mackey products of the form $(\mathcal{M}/p\otimes\mathcal{N}/p)(k)$, where $p$ is a prime number. A usual strategy  to show that a symbol $\{a,b\}_{k/k}\in(\mathcal{M}/p\otimes\mathcal{N}/p)(k)$ vanishes is to consider a finite extension $k'/k$ over which  there exists an element $a'\in\mathcal{M}(k')$ with  $pa'=\res_{k'/k}(a)$.  If $b=N_{k'/k}(b')$, for some $b'\in\mathcal{N}(k')$, then the projection formula \eqref{projectionformula} yields an equality 
\[\{a,b\}_{k/k}=\{a,N_{k'/k}(b')\}_{k/k}= \{\res_{k'/k}(a),b'\}_{k'/k} 
=N_{k'/k}(\{pa',b'\}_{k'/k'})=0.\]  
This trick was used 
in \cite[Proposition 4.3]{Tate1976} 
and refined in the context of Mackey functors by Kahn in  \cite{Kahn92b}.
\end{rem} 
\begin{notn} Let $G$ be a  commutative algebraic group over $k$ and $a\in G(k)$. For every $n\geq 1$ the multiplication-by-$n$ map $G(\overline{k})\xrightarrow{n} G(\overline{k})$ is surjective. Suppose that $G[n] := G(\overline{k})[n] \subset G(k)$. We will denote by $k\left(\frac{1}{n}a\right)$ the smallest finite extension $k'$ of $k$ over which there exists an element $a'\in G(k')$ such that $na=a'$. The assumption $G[n]\subset G(k)$ implies that this is a Galois extension. 
\end{notn}

\subsection{Somekawa $K$-group} For semi-abelian varieties $G_1,\ldots, G_r$ over $k$ the \textbf{Somekawa $K$-group} $K(k;G_1,\ldots,G_r)$ attached to $G_1,\ldots, G_r$ is a quotient of the Mackey product $(G_1\otimes\cdots\otimes G_r)(k)$ (see \cite{Somekawa1990} for the precise definition).  Although our statements will often concern the $K$-group $K(k;G_1,\ldots,G_r)$, all our computations will be at the level of Mackey products, hence we omit the definition of $K(k;G_1,\ldots,G_r)$. We only highlight the  following facts.
\begin{itemize}
\item For every $K/k$ finite there is a surjection,
$(G_1\otimes\cdots\otimes G_r)(K)\twoheadrightarrow K(K;G_1,\ldots,G_r).$
\item When $G_1=\cdots=G_r=\G_m$, there is an isomorphism $K(k;G_1,\ldots,G_r)\simeq K_r^M(k)$ with the Milnor $K$-group of $k$ (cf.\ \cite{Somekawa1990}). 
\item The elements of $K(k;G_1,\ldots,G_r)$ will also be denoted as linear combinations of symbols of the form $\{x_1,\ldots,x_r\}_{K/k}$, where $K/k$ is some finite extension and $x_i\in G_i(K)$ for $i=1,\ldots, r$.
\item The Somekawa $K$-group $K(k;G_1,\ldots,G_r)$ inherits all the properties of the Mackey product $(G_1\otimes\cdots\otimes G_r)(k)$ discussed in the previous subsection.
\end{itemize}

\subsection{Galois symbol map} 
Let $G$ be a semi-abelian variety over $k$ and $p$ be a prime number.  Since we assumed that the field $k$ has characteristic zero, for any finite extension $K/k$ the Kummer sequence 
\[0\to G[p] \to G(\overline{K}) \xrightarrow{p} G(\overline{K}) \to 0\]
is a short exact sequence of $G_K$-modules and hence it induces a connecting homomorphism  
\begin{equation}
\label{Kummer map}
	\delta_{G}: G(K)/p \hookrightarrow H^1(K,G[p]), 
\end{equation}
which is often called the \textbf{Kummer map}. 

\begin{defn}[{cf.\ \cite[Proposition 1.5]{Somekawa1990}}]
\label{symbol map}
Let $G_1, G_2$ be semi-abelian varieties over $k$ and $p$ be a prime. 
By considering $G_1,G_2$ as Mackey functors (as in Example~\ref{MFexs} (1)), 
the \textbf{Galois symbol map} 
\[
s_p:(G_1\otimes  G_2)(k)/p \to H^2(k,G_1[p]\otimes  G_2[p])
\]
is defined by the cup product and the corestriction 
as follows:  
 \[
s_p(\{x,y\}_{K/k}) = \mathrm{Cor}_{K/k}\left(\delta_{G_1}(x)\cup \delta_{G_2}(y)\right).
\]
\end{defn}
%

\smallskip
For semi-abelian varieties $G_1,G_2$ over $k$, 
it is known that 
the Galois symbol  
$s_p$ defined above factors through the Somekawa $K$-group 
$K(k;G_1, G_2)/p$ (\cite[Proposition 1.5]{Somekawa1990}). 
The induced homomorphism $K(k;G_1,G_2)/p \to H^2(k,G_1[p]\otimes G_2[p])$ will also be denoted by $s_p$.
In particular, when $G_1 = G_2 = \mathbb{G}_m$, we have the following commutative diagram
\begin{equation}
\label{classical symbol}
\vcenter{
	\xymatrix{
	(\mathbb{G}_m\otimes\mathbb{G}_m)(k)/p \ar[d]_{s_p}\ar[r]^-{\simeq} &  K_2^M(k)/p \ar[ld]^{g_p} \\
	  H^2(k,\mu_p^{\otimes 2}),
	}
}
\end{equation}
where the map $g_p:K_2^M(k)/p\rightarrow H^2(k,\mu_p^{\otimes 2})$ is the classical Galois symbol, which is an isomorphism by the Merkurjev/Suslin theorem (\cite{Merkurev/Suslin1982})
 (For the isomorphism $(\G_m\otimes\G_m)(k)/p\simeq K_2^M(k)/p$ see Example~\ref{MFexs} (3)). 
Now, we suppose $\mu_p\subset k^{\times}$. In this case the map $g_p$ sends the symbol $\{x,y\}_{k/k}$ to the central simple algebra $(x,y)_p$.  
From \cite[Proposition 4.3]{Tate1976}, we have the following equivalences 
\begin{equation}
\label{Tate}	
\{x,y\}_{k/k} = 0\ \mbox{in $(\G_m\otimes\G_m)(k)/p$} 
\Leftrightarrow (x,y)_p = 0 
\Leftrightarrow x \in N_{k_1/k}(k_1^{\times}), 
\end{equation}
where $k_1 = k(\sqrt[p]{y})$. Note that the last implication $\Leftarrow$ follows from the same argument as in \autoref{vanishingtrick}.

\subsection{Relation to zero-cycles}\label{Kzero} Let $X$ be a smooth projective variety over $k$.  We consider the Chow group of zero-cycles, $CH_0(X)$. Recall that this group has a filtration \[CH_0(X)\supset F^1(X)\supset F^2(X)\supset 0,\] where $F^1(X):=\ker(\deg: CH_0(X) \to \Z)$ is the kernel of the degree map, and $F^2(X):=\ker(\alb_X:F^1(X)\to \mathrm{Alb}_X(k))$ is the kernel of the Albanese map. When $X=C_1\times\cdots\times C_r$ is a product of smooth projective, geometrically  connected curves over $k$ such that $C_i(k)\neq\emptyset$ for $i=1,\ldots, r$, the Albanese kernel $F^2(X)$ has been related to the Somekawa $K$-group attached to the Jacobian varieties $J_1,\ldots, J_r$ of $C_1,\ldots, C_r$ by Raskind and Spiess. Namely, we have the following theorem. 

\begin{theo}[{\cite[Theorem 2.2, Corollary 2.4.1]{Raskind/Spiess2000}}]\label{zerocycles} For $X=C_1\times\cdots\times C_r$ as above there is a canonical isomorphism, 
\[
\displaystyle CH_0(X)\xrightarrow{\simeq}\Z \oplus \bigoplus_{1\leq\nu\leq r}\bigoplus_{1\leq i_1<i_2<\cdots<i_\nu\leq r} K(k;J_{i_1},\ldots, J_{i_\nu}).
\]
\end{theo}
Since each curve $C_i$ has a $k$-rational point, we have 
$CH_0(X) \simeq \Z \oplus F^1(X)$. Hence, we have 
\[
	F^1(X) \xrightarrow{\simeq} \bigoplus_{1\leq\nu\leq r}\bigoplus_{1\leq i_1<i_2<\cdots<i_\nu\leq r} K(k;J_{i_1},\ldots, J_{i_\nu}).
\]
Since we have $K(k;J_i) \simeq J_i(k)$ 
and the Albanese map $\mathrm{alb}_X:F^1(X)\to \operatorname{Alb}_X(k) = J_1(k)\oplus \cdots \oplus J_r(k)$ is surjective, 
we obtain 
\begin{equation}
	\label{F^2}
	F^2(X)  \xrightarrow{\simeq} \bigoplus_{2\leq\nu\leq d}\bigoplus_{1\leq i_1<i_2<\cdots<i_\nu\leq d} K(k;J_{i_1},\ldots, J_{i_\nu})
\end{equation}
when $r\ge 2$.
In particular, when $X=C_1\times C_2$, 
we have an isomorphism $F^2(X)\xrightarrow{\simeq} K(k;J_1,J_2)$, which has the following explicit description. Fix base points $x_0\in C_1(k)$, $y_0\in C_2(k)$. Let $x\in C_1(L)$, $y\in C_2(L)$ where $L/k$ is some finite extension. Let $\pi_{L/k}:X_L\rightarrow X$ be the projection. Then \[\pi_{L/k\star}([x,y]-[x,y_0]-[x_0,y]+[x_0,y_0])\mapsto \{x-x_0,y-y_0\}_{L/k}.\]

\smallskip
\section{Preliminary Computations}

\begin{conv}
\label{conv:p-adic} From now on we assume that $k$ is a finite extension of the $p$-adic field $\Q_p$ 
with absolute ramification index $e_k$, where $p$ is an \textbf{odd} prime. 
We will denote by  $\mathcal{O}_k$ the ring of integers of $k$, $\mathfrak{m}_k$ the maximal ideal of $\mathcal{O}_k$ and $\F$ its residue field. 
For a finite extension $k'/k$, we will denote by  $v_{k'}$ the discrete valuation of $k'$ that extends the one of $k$, and by $\F_{k'}$ the residue field of $k'$.	
We also denote by 
	$\mathcal{O}_{k'}^{\times} = U_{k'}^0$ the group of units in $\mathcal{O}_{k'}$ 
and $U_{k'}^i := 1 + \mathfrak{m}_{k'}^i$ for $i\ge 1$ the higher unit groups. 
\end{conv}

In this section we will obtain  some essential information about the Mackey functor $E/p$ (cf.\ Example~\ref{MFexs}, (2))
for an elliptic curve $E$ over $k$, and we will discuss some necessary ramification theory.

\subsection{Formal groups}
\label{subsect:formal groups}
Let $F$ be a formal group law (which is commutative of dimension one) over $\mathcal{O}_k$. 
The group $F(\mathfrak{m}_k)$ associated to $F$ is denoted by $F(k)$. 
The group 
$F(k)$ has a natural filtration given by 
$F^i(k):=F(\mathfrak{m}^i_k)$, for $i\geq 1$. 
Let $\phi:F\to F'$ be an isogeny of formal groups over $\mathcal{O}_k$ which is given by a formal power series $\phi(T) = a_1T + a_2T^2 + \cdots $ in $\mathcal{O}_k[\![T]\!]$. 
We define a filtration on $F'(k)/\phi(F(k))$ by  
\begin{equation}\label{def:fil}
	\mathcal{J}_k^i := \img\left( F'^i(k) \to F'(k)/\phi(F(k))\right) = \frac{F'^i(k)}{\phi(F(k))\cap F'^i(k)}, 
\end{equation}
for $i\geq 1$. 
Recall that the \emph{height} of $\phi$ is defined to be  the  positive integer $n$ such that $\phi(T) \equiv \psi(T^{p^n}) \bmod \mathfrak{m}_k$ 
for some $\psi\in \mathcal{O}_k[\![T]\!]$ 
whose leading coefficient is a unit in $\mathcal{O}_k$.
The structure of these graded pieces are known as follows:  

\begin{prop}[{\cite[Lemma 2.1.4]{kawachi2002}}]
\label{graded quotients}
Assume that the isogeny $\phi: F\to F'$  
which is given by the power series $\phi(T) = a_1T + a_2T^2 + \cdots$
has height 1 and $F[\phi]:=\ker(\phi) \subset F(k)$. 
Putting $t(\phi) =  v_k(a_1)$ and $t_0(\phi)= t(\phi)/(p-1)$, we have 
\begin{enumerate}[label=$(\alph*)$]
\item If $1\leq i< pt_0(\phi)$ and $i$ is coprime to $p$, then 
$\mathcal{J}_k^i/\mathcal{J}_k^{i+1} \simeq F'^i(k)/F'^{i+1}(k) \simeq \F$. 
\item If $1 \leq i< pt_0(\phi)$ and $i $ is divisible by $p$, then 
$\mathcal{J}_k^i/\mathcal{J}_k^{i+1}=0$.

\item If $i=pt_0(\phi)$, then $\mathcal{J}_k^i/\mathcal{J}_k^{i+1}\simeq \Z/p$.
\item If $i>pt_0(\phi)$, then $\mathcal{J}_k^i = 0$. 
\end{enumerate}
\end{prop}
The above isomorphisms are induced by the standard isomorphism $F'^i(k)/F'^{i+1}(k)\xrightarrow{\simeq} \F$ as in \cite[IV.2, Proposition 6]{serre1979local} and they depend on the choice of a uniformizer $\pi_k$ of $k$. 

A typical example of a height 1 isogeny 
is the multiplication  $[p]:\widehat{\mathbb{G}}_m \to \widehat{\mathbb{G}}_m$ by $p$
	on the multiplicative group $\widehat{\mathbb{G}}_m$. 
This isogeny is given by the power series $[p](T) = pT+ \cdots $ 
and $\widehat{\mathbb{G}}_m[p] = \mu_p$, the group of $p$-th roots of unity. 
In particular, $t([p]) = e_k$ is 
the absolute ramification index of $k$. 
The filtration $\widehat{\mathbb{G}}_m^i(k) = U_k^i$ defines a filtration on $k^{\times}/p = k^{\times}/(k^{\times})^p$ by  
\[
 \overline{U}_k^i :=  \img\left(U_k^i \to k^{\times }/p\right)\quad \mbox{for $i\ge 0$}.
\]
Applying \autoref{graded quotients} to 
$[p]:\widehat{\mathbb{G}}_m \to \widehat{\mathbb{G}}_m$, 
the structure of the graded quotients $\overline{U}^i_k/\overline{U}_k^{i+1}$ is summarized as follows. 
\begin{lem}\label{units} Assume $\mu_p\subset k$. Set $\displaystyle e_0(k)= e_k/(p-1)$ (which is an integer).  
\begin{enumerate}[label=$(\alph*)$]
\item If $0\leq i< pe_0(k)$ and $i$ is coprime to $p$, then $\overline{U}^i_k/\overline{U}_k^{i+1}\simeq
\F_k$.
\item If $0\leq i< pe_0(k)$ and $i$ is divisible by $p$, then 
$\overline{U}^i_k/\overline{U}_k^{i+1}=1$.
\item If $i=pe_0(k)$, then $\overline{U}^i_k/\overline{U}_k^{i+1}\simeq
\Z/p$.
\item If $i>pe_0(k)$, then $\overline{U}_k^i=1$. 
\end{enumerate}
\end{lem}

For the remaining of this subsection we assume that  $\phi: F\to F'$ is a height $1$ isogeny with $F[\phi]\subset F(k)$. For a point $x\in F'(k)$, let $k':=k(\phi^{-1}(x))$ be the smallest extension of $k$ over which there exists a point $y\in F(k')$ such that $\phi(y)=x$.  
Under our assumptions this is a finite Galois extension of $k$ of degree $p$ (cf.\ \autoref{Kummer extension} below). 
Let $G=\Gal(k'/k)$ be the Galois group of $k'/k$. Consider the ramification filtration 
 $(G^\lambda)_{\lambda \ge -1}$ in the upper numbering of $G$ (\cite[IV.3]{serre1979local}). 
Then there is an integer $s$, called the \textbf{jump} of $G$, such that $G^{\lambda}=G$, for every $\lambda \leq s$ and $G^{\lambda}=1$, for every $\lambda > s$ (\cite[IV.3, Theorem]{serre1979local}). 
The ramification of the extension $k'/k$ is described in the following lemma due to Kawachi. 

\begin{lem}[{\cite[Lemma~2.1.5]{kawachi2002}}]
\label{Kummer extension}
For a height $1$ isogeny $\phi: F\to F'$ with $F[\phi]\subset F(k)$, 
and $x\in \mathcal{J}_k^i \setminus \mathcal{J}_k^{i+1}$, 
the ramification of the extension $k' = k(\phi^{-1}(x))/k$ is known as follows:
\begin{enumerate}[label=$(\alph*)$]
	\item If $1\le i < pt_0(\phi)$, then $k'/k$ is a totally ramified extension of degree $p$  with the jump in the ramification filtration of $G := \Gal(k'/k)$ occurring at $s = pt_0(\phi) - i$.  
	\item If $i = pt_0(\phi)$, then $k'/k$ is an unramified extension of degree $p$.
\end{enumerate}
\end{lem}
Note that \autoref{graded quotients} implies that \autoref{Kummer extension} covers all possible cases. Next we further  assume that $\mu_p \subset k^{\times}$.  
The short exact sequences of $G_k$-modules 
$1\to \mu_p \to k^{\times} \xrightarrow{p} k^{\times}\to 1$ and 
$0\to F[\phi] \to F \xrightarrow{\phi} F' \to 0$ give rise to connecting homomorphisms (Kummer maps, cf.\ \eqref{Kummer map})
\[
k^{\times}/p \xrightarrow{\simeq} H^1(k,\mu_p),\quad \mbox{and} \quad F'(k)/\phi(F(k)) \hookrightarrow H^1(k,F[\phi]).
\]
After fixing an isomorphism $F[\phi]\simeq \mu_p$ of 
(trivial) $G_k$-modules, we get noncanonical homomorphisms 
\[
f :F'(k)/\phi(F(k)) \hookrightarrow H^1(k,F[\phi]) \simeq H^1(k,\mu_p) \xleftarrow{\simeq} k^{\times}/p.
\] 
In fact, the induced map is compatible with filtrations and 
the following theorem holds: 

\begin{theo}[{\cite[Theorem 2.1.6]{kawachi2002}}]
\label{filtered isom}
	The map $f:F'(k)/\phi(F(k)) \to k^{\times}/p$ satisfies 
	\[
	f(\mathcal{J}_k^i) = \overline{U}_k^{p(e_0(k)-t_0(\phi)) + i}
	\]
	for any $i\ge 1$, 
	where $e_0(k) = e_k/(p-1)$ and $t_0(\phi)$ is defined in Proposition \ref{graded quotients}.
\end{theo}

We close this subsection by defining a Mackey functor by the higher unit groups $\overline{U}_k^i$. 

\begin{defn}\label{unitsMF} Let $i\geq 0$. We define the sub-Mackey functor $\overline{U}^i$, of $\G_m/p$ as follows. If $L$ is a finite extension of $k$, then 
$\overline{U}^i(L):=\overline{U}_L^{ie(L/k)}.$ For a finite extension $F/L$, the norm $N_{F/L}$ and restriction maps $\res_{F/L}$ are induced by the ones on $\G_m$. 
\end{defn}

\subsection{Elliptic Curves}
 Let $E$ be an elliptic curve over $k$ with good reduction. 
 Let $\mathcal{E}$ be the N\'{e}ron model of $E$, which is an abelian scheme over $\Spec(\mathcal{O}_k)$. 
Let $\overline{E}:=\mathcal{E}\otimes_{\mathcal{O}_k}\F$ be the reduction of $E$, and $\widehat{E}$ be the formal group  of $E$. 
 We have a short exact sequence of abelian groups,
\begin{eqnarray}\label{ses-1}0\rightarrow \hat{E}(\mathfrak{m}_k)\stackrel{j}{\longrightarrow}E(k)\stackrel{r}{\longrightarrow}\overline{E}(\F)\rightarrow 0,\end{eqnarray}
where 
$r:E(k) \to \overline{E}(\F)$ is the reduction map. 
This induces a \textbf{short exact sequence of Mackey functors}, 
\begin{eqnarray}\label{ses-2}0\rightarrow 
		\widehat{E}\rightarrow E\rightarrow [E/\widehat{E}]\rightarrow 0,\end{eqnarray} 		where $\widehat{E}$ is the Mackey functor given by $K\mapsto \widehat{E}(K):=\widehat{E}(\mathfrak{m}_K)$, and $[E/\widehat{E}]$ is a Mackey functor defined by $L\mapsto \overline{E}(\F_L)$. For a finite extension $L/K$, the restriction $\res_{L/K}:[E/\hat{E}](K)\rightarrow [E/\hat{E}](L)$ is the usual restriction, $\overline{E}(\F_K)\xrightarrow{\res_{L/K}} \overline{E}(\F_L)$, while the norm $N_{L/K}:[E/\hat{E}](L)\rightarrow [E/\hat{E}](K)$ is the map $\overline{E}(\F_L)\xrightarrow{e(L/K)\cdot N_{L/K}} \overline{E}(\F_K)$. The fact that $[E/\hat{E}]$ is a Mackey functor has been shown by Raskind and Spiess (\cite[p.~15]{Raskind/Spiess2000}).

		
From now on  we will drop the notation $[p]:\widehat{E}(k)\rightarrow\widehat{E}(k)$ and will denote it simply by $p:\widehat{E}(k)\rightarrow\widehat{E}(k)$. Recalling from Section \ref{subsect:formal groups}, 
the group 
$\widehat{E}(k)=\widehat{E}(\mathfrak{m}_k)$ has a natural filtration given by $\widehat{E}^i(k):=\widehat{E}(\mathfrak{m}^i_k)$, for $i\geq 1$. 
\begin{defn}\label{formalfil}  Following \eqref{def:fil}, we define a filtration on the group $\widehat{E}(k)/p = \widehat{E}(k)/p\widehat{E}(k)$  by \[\mathcal{D}_k^i := \displaystyle\frac{\widehat{E}^i(k)}{p\widehat{E}(k)\cap \widehat{E}^i(k)},\;\;i\geq 1.\] 
\end{defn}

 In the case when $\widehat{E}[p]\subset E(k)$, we can decompose $E/p$ using the Mackey functors $\G_m/p$ and $\overline{U}^i$ (cf.\ Definition~\ref{unitsMF}). This decomposition depends on the reduction type of $E$, which we review in the following two subsections.

\subsection{Elliptic curves with good ordinary reduction}\label{ordinarysection} We first consider the case when $E$ has \textbf{good ordinary reduction}; that is, $\overline{E}$ is an ordinary elliptic curve over $\F$.  
In this case the $G_k$-module $E[p]$ has a one-dimensional $G_k$-invariant submodule. Namely, we have the {connected-\'{e}tale} short exact sequence of $G_k$-modules, 
\begin{equation}\label{ses1}
0\rightarrow 
		E[p]^\circ\rightarrow E[p]\rightarrow E[p]^{et}\rightarrow 0,
\end{equation}
where $E[p]^\circ:=\hat{E}[p]$ are the $[p]$-torsion points of the formal group $\hat{E}$ of $E$.  For more details on the connected-\'{e}tale exact sequence we refer to \cite[Section 8]{Stix2009}. After a finite unramified extension $k'/k$ the short exact sequence \eqref{ses1} becomes,
\begin{eqnarray}\label{ses2}&&0\rightarrow\mu_{p}\rightarrow E[p]\rightarrow\Z/p\rightarrow 0.\end{eqnarray}  This follows because after base change to the maximal unramified extension $k^{nr}$ of $k$ the formal group $\widehat{E}^{nr}$ becomes isomorphic to the multiplicative group $\widehat{\G}_m$  (cf.\ \cite[Lemma 4.27]{Mazur1972}). 
As the multiplication $p:\widehat{E} \to \widehat{E}$ has height $1$, Theorem \ref{filtered isom}  now leads to the following proposition.

\begin{prop}[{\cite[Theorem 2.1.6]{kawachi2002}, see also\ \cite[Proposition 3.12]{Gazaki/Leal2018}}]\label{formalgroup}   Let $E$ be an elliptic curve over $k$ with good ordinary reduction, and $\widehat{E}$ be its formal group. Assume that $\widehat{E}[p]\subset\widehat{E}(k)$  and $\mu_p\subset k^\times$. Then, we have an isomorphism  $f:\widehat{E}(k)/p\xrightarrow{\simeq} \overline{U}_k^1\simeq \overline{U}_k^0$, which  
		 satisfies $f(\mathcal{D}_k^i) = \overline{U}_k^i$ for $i \ge 1$.
\end{prop} 
Assume we are in the situation of \autoref{formalgroup}. In particular, $\mu_p\subset k^\times$ and after a finite unramified extension \eqref{ses2} holds. 
Next we consider the short exact sequence \eqref{ses-1} and apply the right exact functor $\otimes_\Z\Z/p$. Using \autoref{formalgroup} we get an exact sequence of $\F_p$-vector spaces,
\[\overline{U}^0_k\stackrel{j}{\longrightarrow} E(k)/p\stackrel{r}{\longrightarrow} \overline{E}(\F)/p\rightarrow 0.\]
The map $j$ is not always injective. Namely, there is a unit $u\in \mathcal{O}_k^\times/\mathcal{O}_k^{\times p}$, known as the \textit{Serre-Tate} parameter of $E$, that generates the kernel. For more details see \cite[Section 3.1]{Gazaki/Leal2018}. All we need in this article is that there is an exact sequence  of Mackey functors,
\begin{eqnarray}\label{ses4}\overline{U}^0\rightarrow E/p\rightarrow [E/\widehat{E}]/p\rightarrow 0,\end{eqnarray} where $\overline{U}^0, [E/\widehat{E}]$ are the Mackey functors defined previously.   
%
%

\subsection{Elliptic curves with good supersingular reduction}\label{sssetup} Lastly, we  consider the case when $E$ has \textbf{good supersingular reduction}. 
This means that the elliptic curve $\overline{E}$ has no $p$-torsion, and hence we get a surjection of Mackey functors,
\begin{eqnarray}\label{ssing0}
&&\widehat{E}/p\xrightarrow{j} E/p\rightarrow 0. 
\end{eqnarray}

For the rest of this subsection we assume that $E[p]\subset E(k)$. Let $e_0(k)=\dfrac{e_k}{p-1}$.  
 Suppose $y_0, y_1$ are two generators of $\widehat{E}[p]\simeq\Z/p \oplus\Z/p$ such that \[v_k(y_0)=\max\{v_k(y):y\in\widehat{E}[p], y\neq 0\}.\] From now on we will denote by $t_0(k):=v_k(y_0)$. We have a decomposition of Mackey functors (see \cite[Proof of Theorem 4.1]{Hiranouchi2014}),
\begin{eqnarray}\label{MF}E/p\simeq\overline{U}^{p(e_0(k)-t_0(k))}\oplus\overline{U}^{pt_0(k)},\end{eqnarray} where $\overline{U}^{p(e_0(k)-t_0(k))}$, $\overline{U}^{pt_0(k)}$ are Mackey functors defined as in Definition~\ref{unitsMF}. We recall how this decomposition is obtained. The main reference for what follows is \cite{kawachi2002}.  Consider the isogenous elliptic curve $E':=E/\langle y_0\rangle$ and the isogeny 
$E\xrightarrow{\phi} E'$, and its dual, $E'\xrightarrow{\check{\phi}} E$, which are both defined over $k$. 
As $[p] = \check{\phi} \circ \phi:E\to E$ 
has height $2$, 
these isogenies correspond to isogenies of formal groups,  
$\widehat{E}\xrightarrow{\phi} \widehat{E}'$, $\widehat{E}'\xrightarrow{\check{\phi}} \widehat{E}$, which are both of height $1$. We have a splitting short exact sequence of $\F_p$-vector spaces, 
\begin{eqnarray}\label{SES}&&0\rightarrow\frac{E'(k)}{\phi(E(k))}\xrightarrow{\check{\phi}}\frac{E(k)}{p}\xrightarrow{\varepsilon}\frac{E(k)}{\check{\phi}(E'(k))}\rightarrow 0,\end{eqnarray} where $\varepsilon$ is just the projection. 
By \cite[Corollary 2.3]{Hiranouchi/Hirayama2013}, 
we have $t_0(\phi) = t_0(k)$.  
From \cite[Theorem 3.2.6 (2)]{kawachi2002}, we have isomorphisms 
\[\frac{E'(k)}{\phi(E(k))} \simeq \frac{\widehat{E}'(k)}{\phi(\widehat{E}(k)))} \underset{f}{\xrightarrow{\simeq}} \overline{U}_k^{p(e_0(k)-t_0(k))} \simeq \overline{U}_k^{p(e_0(k)-t_0(k))+1},\] where 
the middle isomorphism $f$ follows from \autoref{filtered isom}, and the last one follows from \autoref{units},~$(b)$.  
 It is clear that the dual isogeny $\check{\phi}$ identifies $E$ with the quotient $E'/\langle\phi(y_1)\rangle$. 
 By \cite[Corollary 2.3]{Hiranouchi/Hirayama2013} again, 
	we have $t_0(\check{\phi}) =v_k(\phi(y_1))$. 
	The leading coefficient of 
	$[p] = \check{\phi}\circ \phi$ equals to $p$ so that 
	$t(\phi) + t(\check{\phi}) = e_k$ and hence 
	$t_0(k) + t_0(\check{\phi} ) = e_0(k)$. 
	Then the same theorem \cite[Theorem 3.2.6 (2)]{kawachi2002} gives an isomorphism, 
 \[\frac{E(k)}{\check{\phi}(E'(k))} \simeq \frac{\widehat{E}(k)}{\check{\phi}(\widehat{E}'(k))}\underset{f}{\xrightarrow{\simeq}}\overline{U}_k^{p(e_0(k)-t_0(\check{\phi}))+1} = \overline{U}_k^{pt_0(k)+1}.\]

 \subsection*{Behavior with respect to the filtration of the formal group} 
\begin{notn} Consider the filtration $\mathcal{D}_k^i =\displaystyle\frac{\widehat{E}^i(k)}{p\widehat{E}(k)\cap \widehat{E}^i(k)}$, $i\geq 1$ (cf.\ Definition~\ref{formalfil}). We define additionally $\mathcal{F}_k^i:=\displaystyle\frac{\widehat{E}^i(k)}{\check{\phi}(\widehat{E}'(k))\cap \widehat{E}^i(k)}$, and $\mathcal{G}_k^i:=\displaystyle\frac{\widehat{E}^{'i}(k)}{\phi(\widehat{E}(k))\cap \widehat{E}^{'i}(k)}$, for $i\geq 1$.  
\end{notn} 
Note that since the isogenies $\phi, \check{\phi}$ are of height $1$, \autoref{graded quotients} applies for the quotients $\mathcal{F}^i_k/\mathcal{F}^{i+1}_k$, $\mathcal{G}^i_k/ \mathcal{G}^{i+1}_k$.

\begin{rem}\label{filtration1} It is clear that the map $\varepsilon:\mathcal{D}^1_k\rightarrow\mathcal{F}^1_k$ preserves the filtration.  The following lemma gives a complete description of the quotients $\mathcal{D}^i_k/\mathcal{D}^{i+1}_k$ for $i\geq 1$. 
\end{rem}

\begin{lem}\label{filtration2} Let $i\geq 1$ be an integer. 
Suppose that we have $e_0(k) - t_0(k) = pt_0(k)$. 
Then the following are true for the quotient $\mathcal{D}^i_k/\mathcal{D}^{i+1}_k$.
\begin{enumerate}[label=$(\alph*)$]
\item If $i<p(e_0(k)-t_0(k))$ and $(i,p)=1$, then $\varepsilon$ induces an isomorphism, \[\varepsilon:\mathcal{D}^i_k/\mathcal{D}^{i+1}_k\xrightarrow{\simeq}\mathcal{F}^i_k/\mathcal{F}^{i+1}_k\simeq\F.\] Moreover, both quotients are isomorphic to $\widehat{E}^i(k)/\widehat{E}^{i+1}(k)$. 
\item 
If $i<p(e_0(k)-t_0(k))$ and $i=pj$, for some $j\geq 1$, then $\check{\phi}$ induces an isomorphism, \[\check{\phi}:\mathcal{G}^j_k/\mathcal{G}^{j+1}_k\xrightarrow{\simeq}\mathcal{D}^i_k/\mathcal{D}^{i+1}_k.\] Moreover, we have an isomorphism $\F \simeq \widehat{E}^{'j}(k)/\widehat{E}^{'j+1}(k) \simeq \mathcal{D}^i_k/\mathcal{D}^{i+1}_k$ 
if $(j,p) = 1$. 
\item If $i=p(e_0(k)-t_0(k))$, then we have a short exact sequence of $\F_p$-vector spaces, \[0\rightarrow \frac{\mathcal{G}_k^{e_0(k)-t_0(k)}}{\mathcal{G}_k^{e_0(k)-t_0(k)+1}}\xrightarrow{\check{\phi}}\frac{\mathcal{D}_k^{p(e_0(k)-t_0(k))}}{\mathcal{D}_k^{p(e_0(k)-t_0(k))+1}}\xrightarrow{\varepsilon}\frac{\mathcal{F}_k^{p(e_0(k)-t_0(k))}}{\mathcal{F}_k^{p(e_0(k)-t_0(k))+1}}\rightarrow 0.\] 
Moreover, we have isomorphisms, $\displaystyle\frac{\mathcal{G}_k^{e_0(k)-t_0(k)}}{\mathcal{G}_k^{e_0(k)-t_0(k)+1}}\simeq\frac{\mathcal{F}_k^{p(e_0(k)-t_0(k))}}{\mathcal{F}_k^{p(e_0(k)-t_0(k))+1}}\simeq\Z/p$.
\item If $i>p(e_0(k)-t_0(k))$, then $\mathcal{D}_k^i=0$.
\end{enumerate}
\end{lem} 
\begin{proof}
All the parts will follow from \autoref{graded quotients} and \cite[
Lemma 2.6]{Hiranouchi/Hirayama2013}. 
For the first claim \autoref{graded quotients} ($a$) gives an isomorphism, $\mathcal{F}^i_k/\mathcal{F}^{i+1}_k\simeq\F$. At the same time, we have a surjective map of finite dimensional $\F_p$-vector spaces, $\mathcal{D}^i_k/\mathcal{D}^{i+1}_k\xrightarrow{\varepsilon}\mathcal{F}^i_k/\mathcal{F}^{i+1}_k\rightarrow 0$. Moreover, the projection $\widehat{E}^i(k)\rightarrow\mathcal{D}^i_k/\mathcal{D}^{i+1}_k$ induces a surjection $\widehat{E}^i(k)/\widehat{E}^{i+1}(k)\rightarrow\mathcal{D}^i_k/\mathcal{D}^{i+1}_k\rightarrow 0$. Since the quotient $\widehat{E}^i(k)/\widehat{E}^{i+1}(k)$ is isomorphic to $\F$ (\cite[Chapter IV, Proposition 3.2]{Silverman2009}), the claim follows. 

For the second claim, 
the proof of  \cite[Lemma 2.1.4, (1)]{kawachi2002} shows an isomorphism $\mathcal{G}^j_k/\mathcal{G}^{j+1}_k\xrightarrow{\simeq}\mathcal{D}^i_k/\mathcal{D}^{i+1}_k$ induced by $\check{\phi}$. 
If $p \nmid j$, then 
$\F\simeq \widehat{E}'^{j}(k)/\widehat{E}'^{j+1}(k) \simeq \mathcal{G}^j_k/\mathcal{G}^{j+1}_k$ by \autoref{graded quotients} $(a)$. 

Next, suppose $i=p(e_0(k)-t_0(k))$. The isomorphisms $\displaystyle\frac{\mathcal{G}_k^{e_0(k)-t_0(k)}}{\mathcal{G}_k^{e_0(k)-t_0(k)+1}}\simeq\frac{\mathcal{F}_k^{p(e_0(k)-t_0(k))}}{\mathcal{F}_k^{p(e_0(k)-t_0(k))+1}}\simeq\Z/p$ follow by  \autoref{graded quotients} ($c$). The short exact sequence follows easily by the short exact sequence $0\rightarrow\mathcal{G}_k^1\xrightarrow{\check{\phi}}\mathcal{D}_k^1\xrightarrow{\varepsilon}\mathcal{F}_k^1\rightarrow 0$, after restricting to $\mathcal{D}_k^{p(e_0(k)-t_0(k))}$.  

Finally, the claim $(d)$ follows by \autoref{graded quotients} ($d$). 

\end{proof}

\subsection{Computing ramification jumps}\label{unramified section}

In the remaining two subsections we focus on the following special case. We consider an elliptic curve $E$ with \textbf{good supersingular} reduction over a finite  \textbf{unramified} extension $k/\mathbb{Q}_p$, that is, $e_k = 1$. 
In this situation, 
for the mod $p$ Galois representation 
\[\rho: \Gal(\overline{k}/k) \to \operatorname{Aut}(E[p])\simeq GL_2(\F_p),\] 
the image of the inertia subgroup by $\rho$ is known to be cyclic of order $p^2-1$ 
(\cite[Proposition~12]{Serre1972}). 
This implies that 
\begin{itemize}
	\item the extension $L := k(E[p])/k$ corresponding to the kernel of $\rho$ has ramification index $e_{L/k} = p^2-1$, and 
	\item the inertia subgroup of $\Gal(L/k)$ is cyclic.
\end{itemize}

\begin{lem}\label{valuation}
For every non-zero $x\in \widehat{E}[p]$, we have $v_L(x) = 1$.
\end{lem}
\begin{proof}
	Since the valuation of the $p$-th coefficient, $a_p$, of multiplication by $p$ in the formal group $\widehat{E}$ satisfies $v_k(a_p)\ge pe/(p+1) = p/(p+1)$, we have 
	$\displaystyle v_L(x) = e_{L/k}\frac{1}{p^2-1} = 1$  (cf.\ \cite[Theorem~3.10.7]{Katz1973}). 
\end{proof}
From \eqref{MF}, we obtain a decomposition, 
\begin{eqnarray}\label{decompose}&&E(L)/p\simeq 
\overline{U}_L^{p(e_0(L) - t_0(L)) } \oplus \overline{U}_L^{pt_0(L)}
 = \overline{U}_L^{p^2+1}\oplus\overline{U}_L^{p+1},
 \end{eqnarray}
since $t_0(L)=\max\{v_L(x):0\neq x\in\widehat{E}[p]\} = 1$ 
and $e_0(L) = p+1$. 

\smallskip
Next we make a simplifying assumption that the extension $L=k(E[p])/k$ is \textbf{totally ramified} of degree $p^2-1$. We consider the restriction \[\res_{L/k}:\mathcal{D}^1_{k}\rightarrow\mathcal{D}_L^{p^2-1},\] which is an injective map of $\F_p$-vector spaces, since  $[L:k]$  is coprime to $p$ (cf.\ \autoref{normrestr}). Let $x\in\mathcal{D}^1_{k}$. We will identify $x$ with its image, $\res_{L/k}(x)$. We are interested in obtaining concrete information about the ramification of the finite Galois extension $L\left(\frac{1}{p}x\right)$. This will be achieved in \autoref{resdec}, but we first need some reminders from Kummer theory.


Let $F/k$ be an arbitrary finite extension such that $\mu_p\subset F^\times$. 
Let $a\in F^{\times}/p = F^\times/(F^{\times})^p$. Consider the Kummer extension $F_1=F(\sqrt[p]{a})$. 
Let $G=\Gal(F_1/F)$ be the Galois group of $F_1/F$. 
In this setting, \autoref{Kummer extension} now reads as follows:

\begin{lem}
\label{jump} 
Let $a\in F^\times/p$. Consider the Galois extension  $F_1=F(\sqrt[p]{a})$.  
\begin{enumerate}[label=$(\alph*)$]
\item If $a\in\overline{U}_{F}^i\setminus\overline{U}_{F}^{i+1}$, for some $0\le i<pe_0(F)$, then $F_1/F$ is a totally ramified extension of degree $p$ with the jump at $s=pe_0(F)-i$. 
\item If $a \in \overline{U}_F^{pe_0(F)}\setminus \overline{U}_{F}^{pe_0(F)+1}$, 
then the extension $F_1/F$ is unramified extension of degree $p$. 
\end{enumerate}
\end{lem}

Next, consider the Herbrand $\psi_{F_1/F}$ function defined as follows, 
\[\psi_{F_1/F}(x)=\left\{
\begin{array}{cc}
x,& x\leq s\\
s+p(x-s),& x>s.
\end{array}\right.\]

 We will need the following facts about the norm map $N_{F_1/F}$. 
\begin{prop}[{cf.\ \cite[V.3, Corollaries 2 \& 3]{serre1979local}}]\label{psi}
\begin{enumerate}
\item For every integer $1\leq i<s$, the norm map $N_{F_1/F}:F_1^\times\rightarrow F^\times$ induces a surjection, \[\frac{U_{F_1}^{i}}{U_{F_1}^{i+1}}\xrightarrow{N_{F_1/F}}\frac{U_{F}^i}{U_{F}^{i+1}}\rightarrow 0.\]
\item For every integer $i>s$, the norm map  induces a surjection  \[U_{F_1}^{\psi_{F_1/F}(i)}\xrightarrow{N_{F_1/F}} U_{F}^i\rightarrow 0.\]  
\end{enumerate} 
\end{prop}


\smallskip
We now come back to the element $x\in \mathcal{D}_k^1$ considered before the Kummer theory aside. We think of the element $x$ as lying in $\mathcal{D}_L^{p^2-1}$ under the restriction map. 
The following lemma gives us how $x$ decomposes under the decomposition \eqref{decompose}. We note that this lemma is a key computation that will be used several times in the next section. 

\begin{lem}\label{resdec} Let $x\in\mathcal{D}_k^1$. Consider the finite extension $L_0=L\left(\frac{1}{p}x\right)$.
\begin{enumerate}
\item Suppose $x\in \mathcal{D}_k^1\setminus\mathcal{D}_k^2$. Under the decomposition \eqref{decompose} the image of $x$ in $E(L)/p=\mathcal{D}_L^1$ can be written as $(x_0,x_1)$  with $x_0\in\overline{U}_L^{p^2+p}$ and $x_1\in\overline{U}_L^{p+p^2-1}\setminus\overline{U}_L^{p + p^2}$. In particular, we have a tower of finite extensions $L_0\supset L_1\supset L$, with $L_1\supset L$ unramified and $L_0\supset L_1$  totally ramified of degree $p$. The jump in the ramification filtration of $\Gal(L_0/L_1)$ happens at $s=1$. 
\item Suppose $x\in \mathcal{D}_k^2$. Then, $L_0=L$.
\end{enumerate}
\end{lem}
\begin{proof} In what follows we will denote by $\oplus$ the addition law given by the formal group $\widehat{E}_L$. Moreover, we will fix $\pi$, $\pi_L$ uniformizer elements of $k, L$ respectively. 

 We first prove (2). If $x\in \mathcal{D}_k^2$, then $v_L(x)\geq 2(p^2-1)$, and hence its image in $\mathcal{D}_L^1$ lies in $\mathcal{D}_L^{2(p^2-1)}$. Since $2(p^2-1)>p^2$, 
the claim (2) follows  by \autoref{filtration2} $(d)$, 
since $\mathcal{D}_L^{2(p^2-1)}= 0$. 

We next prove (1).  Following the notation from \eqref{SES}, for every finite extension $L'/k$ we have a  splitting short exact sequence of $\F_p$-vector spaces,
\[0\rightarrow\frac{\widehat{E}^{'1}(L')}{\phi(\widehat{E}^1(L'))}\xrightarrow{\check{\phi}}\frac{\widehat{E}^1(L')}{p}\xrightarrow{\varepsilon}\frac{\widehat{E}^1(L')}{\check{\phi}(\widehat{E}^{'1}(L'))}\rightarrow 0.\]  We can therefore decompose the image of $x$ in $ \dfrac{\widehat{E}^1(L)}{p}$ as 
$x=\check{\phi}(w_0)\oplus w_1$, where $w_0$ is the class of some $\tilde{w}_0\in \widehat{E}^{'1}(L)$ mod $\phi(\widehat{E}^1(L))$ and $w_1$ is the class of some $\tilde{w}_1\in\widehat{E}^1(L)$ mod $\check{\phi}(\widehat{E}^{'1}(L))$. Note that the same decomposition holds also over any finite extension $L'/L$. This implies that there exists an element $y\in\widehat{E}^1(L')$ such that $x=py$ if and only if there exist elements $\tilde{z_0}\in \widehat{E}^1(L')$ and $\tilde{z_1}\in \widehat{E}^{'1}(L')$ such that $\tilde{w_0}=\phi(\tilde{z_0})$ and $\tilde{w_1}=\check{\phi}(\tilde{z_1})$. This means that the extension $L\left(\frac{1}{p}x\right)$ is precisely the compositum of the extensions $L_1=L(\phi^{-1}(\tilde{w}_0))$ and $L_2=L(\check{\phi}^{-1}(\tilde{w}_1))$  (for this notation see  paragraph preceding \autoref{Kummer extension}).

 Next we analyze the extensions $L_1, L_2$. We start with $L_1=L(\phi^{-1}(\tilde{w}_0))$ and we claim that the extension $L_1/L$ is unramified. It is enough to show that we are in the set-up of \autoref{Kummer extension} ($b$), which we apply for the height $1$ isogeny $\phi:\widehat{E}\rightarrow\widehat{E}'$.  Recall from the discussion in Section~\ref{sssetup} and \autoref{valuation} that we have an equality $t_0(\phi)=t_0(L)=1$. Thus, it suffices to show that $v_L(\tilde{w}_0)\geq p$. 
 Since $x$ is the restriction of an element from $\mathcal{D}_k^1\setminus\mathcal{D}_k^2$, $v_L(x)=p^2-1$. This yields $v_L(\check{\phi}(\tilde{w}_0))\geq p^2-1$. 
Let $i$ be such that $\tilde{w}_0\in\widehat{E}^{'i}(L)\setminus\widehat{E}^{'i+1}(L)$. Since $\check{\phi}$ is a height $1$ isogeny, it follows that 
\[
v_L(\check{\phi}(\tilde{w}_0)) = pi
\] 
(cf.\ \cite[Lemma 2.2]{Hiranouchi/Hirayama2013}). 
In particular, it is divisible by $p$. 
Combining the two relations, we conclude that $i = v_L(\tilde{w}_0)\ge p$ as required. Next consider the isomorphism \[\displaystyle\frac{\widehat{E}^{'1}(L)}{\phi(\widehat{E}^1(L))}\xrightarrow{\simeq}\overline{U}_L^{p^2+1}\] and let $x_0$ be the image of $w_0$ under this isomorphism. The above computation together with \autoref{filtered isom} give us that $x_0\in\overline{U}_L^{p^2+p}$. 

Next we consider the extension $L_2=L(\check{\phi}^{-1}(\tilde{w}_1))$, and let $\tilde{z}_1\in \widehat{E}'(L_2)$ be such that $\check{\phi}(\tilde{z}_1)=\tilde{w}_1$. We apply \autoref{Kummer extension} ($a$) for the height $1$ isogeny $\check{\phi}:\widehat{E}'\rightarrow\widehat{E}$. By the discussion in \autoref{sssetup} it follows that \[t_0(\check{\phi})=e_0(L)-t_0(L)=p+1-1=p.\] Thus, in order to show that we are in the set-up of \autoref{Kummer extension} ($a$), it is enough to verify that $v_L(w_1)<p^2$, which follows directly by the previous paragraph and the equality $v_L(x)=p^2-1$. We conclude that the extension $L_2/L$ is totally ramified of degree $p$ and \autoref{Kummer extension} ($a$) gives that the jump of the ramification filtration of $\Gal(L_2/L)$ occurs at $s=pt_0(\check{\phi})-(p^2-1)$. Moreover, if $x_1$ is the image of $w_1$ under the isomorphism $\widehat{E}(L)/\check{\phi}(\widehat{E}'(L))\simeq\overline{U}_L^{p+1}$, then \autoref{filtered isom} implies that $x_1\in\overline{U}_L^{p+p^2-1}\setminus\overline{U}_L^{p^2+p}$. 

We conclude the proof by noticing that the extensions $L_1/L$, $L_2/L$ are totally disjoint, one being unramified and the other totally ramified of degree $p$. Thus, for the compositum $L\left(\frac{1}{p}x\right)=L_1\cdot L_2$ we have an isomorphism $\Gal(L_1 L_2/L_1)\simeq\Gal(L_2/L)$. 

\end{proof}

\subsection{Galois action on graded quotients}
 
We close this section with a technical computation that will be used in the proofs of theorems \ref{main2} and \ref{main3}. We continue working in the set-up of Section~\ref{unramified section}. Namely, $E$ is an elliptic curve over a finite unramified extension $k$ of $\Q_p$ and the extension $L=k(E[p])/k$ is totally ramified of degree $[L:k]=e(L/k)=p^2-1$. We want to understand how the Galois group $\Gal(L/k)$ acts on the graded quotients $\mathcal{D}_L^i/\mathcal{D}_L^{i+1}$. This will be achieved in \autoref{filtration3}. Before that we start with a preliminary discussion about tame extensions and formal groups.

Suppose $k/\Q_p$ is finite and   $L/k$ is a totally ramified extension of degree coprime to $p$ (that is, $L/k$ is a tame extension). Fix a uniformizer $\pi$ of $L$. Let $G_0=\Gal(L/k)$ be the Galois group. Because $L/k$ is tame,  $G_0$ is  cyclic (\cite[IV.2, Corollary 1]{serre1979local}). Let $\sigma$ be a generator of $G_0$. We have an injection, $f:G_0\hookrightarrow U_L^0/U^1_L$ given by $\displaystyle\sigma^m\mapsto \frac{\sigma^m(\pi)}{\pi}$. Since there is an isomorphism $U_L^0/U^1_L\simeq\F_L^\times$, $f$ identifies $G_0$ with a subgroup of $\F_L^\times$. Set $\overline{u}=\sigma(\pi)/\pi\in\F_L^\times$. The following lemma shows how $\sigma$ acts on the quotient $U_L^i/U^{i+1}_L\simeq\F_L$, when $i\geq 1$. 

\begin{lem}\label{sigmaaction} Let $i\geq 1$. Then $\sigma$ induces an automorphism, $\overline{\sigma}: U_L^i/U^{i+1}_L\rightarrow U_L^i/U^{i+1}_L$, which is given by multiplication by $\overline{u}^i$. 
\end{lem}
\begin{proof}
Since $U_L^i/U^{i+1}_L\simeq\F_L$, and $L/k$ is totally ramified, the map  $\overline{\sigma}$ is an $\F_L$-linear map, $\overline{\sigma}:\F_L\rightarrow \F_L$, and hence it is given by multiplication by a non-zero scalar $c\in\F_L^\times$. We will show that $c=\overline{u}^i$. For every $i\geq 1$, we have an isomorphism of groups,
\[\frac{U_L^i}{U^{i+1}_L}=\frac{\widehat{\G}_m(\mathfrak{m}_L^i)}{\widehat{\G}_m(\mathfrak{m}_L^{i+1})}\simeq\frac{\widehat{\G}_a(\mathfrak{m}_L^i)}{\widehat{\G}_a(\mathfrak{m}_L^{i+1})}=\frac{(\pi^i)}{(\pi^{i+1})},\] (cf.\ \cite[Chapter IV, Proposition~3.2]{Silverman2009}). It is immediate that the isomorphism is $G_0$-equivariant. Thus, it is enough to see how $\overline{\sigma}$ acts on the quotient $\displaystyle\frac{(\pi^i)}{(\pi^{i+1})}$. But this is immediate, since \[\pi^i\mapsto\sigma(\pi^i)=\left(\frac{\sigma(\pi)}{\pi}\right)^i\pi^i=\overline{u}^i\pi^i,\] and the class of $\pi^i$ corresponds to $1$ under the isomorphism $\displaystyle\frac{(\pi^i)}{(\pi^{i+1})}\simeq\F_L$. 

\end{proof}

 Note that there was nothing special about the multiplicative group in \autoref{sigmaaction}. Namely, according to \cite[Chapter IV, Proposition 3.2]{Silverman2009}, if $\mathcal{F}$ is any formal group over $\mathcal{O}_k$, then for each $i\geq 1$ the map   $\mathcal{F}(\mathfrak{m}^i_k)/\mathcal{F}(\mathfrak{m}^{i+1}_k)\rightarrow\widehat{\G}_a(\mathfrak{m}^i_k)/\widehat{\G}_a(\mathfrak{m}^{i+1}_k)$ induced by the identity map on sets is an isomorphism of groups. Thus, if $L/k$ is a tame extension, then the induced map $\overline{\sigma}: \mathcal{F}(\mathfrak{m}^i_L)/\mathcal{F}(\mathfrak{m}^{i+1}_L)\rightarrow \mathcal{F}(\mathfrak{m}^i_L)/\mathcal{F}(\mathfrak{m}^{i+1}_L)$, which is obviously $G_0$-equivariant, is given by scalar multiplication by $\overline{u}^i=(\sigma(\pi)/\pi)^i$. 
 
  We are now going to apply this to our elliptic curve $E$ satisfying the assumptions of the beginning of this subsection. 
   If $G_0=\Gal(L/k)=\langle\sigma\rangle$, then $\sigma$ induces a $\F_L$-linear automorphism
\[\overline{\sigma}:\widehat{E}^1(L)/\widehat{E}^{2}(L)\xrightarrow{\cdot\overline{u}}\widehat{E}^1(L)/\widehat{E}^{2}(L)
\] of exact order $p^2-1$. In fact, the Galois group $G_0$ permutes the $p$-torsion points. 
Next we see how the $G_0$-action behaves under the decomposition \eqref{decompose}. Following the notation of Section~\ref{sssetup} we may take $E'=E_L/\langle y_0\rangle$, where $y_0$ is some fixed torsion point. Consider the isogenous elliptic curve $E^{'\sigma}=E_L/\langle \sigma(y_0)\rangle$ and the isogenies $E_L\xrightarrow{\phi^\sigma} E^{'\sigma}$, $E^{'\sigma}\xrightarrow{\check{\phi^\sigma}} E_L$. It is clear that we have an equality $\phi^\sigma=\sigma\circ\phi\circ\sigma^{-1}$. We therefore get a commutative diagram,
\[\begin{tikzcd}
	\widehat{E}'(L)/\phi(\widehat{E}_L(L))\ar{r}{\check{\phi}}\ar{d}{\sigma}& \widehat{E}_L(L)/p \ar{r}{\varepsilon}\ar{d}{\sigma} & \widehat{E}_L(L)/\check{\phi}(\widehat{E}'(L))\ar{d}{\sigma} \\
\widehat{E}^{'\sigma}(L)/\phi^\sigma(\widehat{E}_L(L))\ar{r}{\check{\phi^\sigma}} & \widehat{E}_L(L)/p \ar{r}{\varepsilon^\sigma} & \widehat{E}_L(L)/\check{\phi^\sigma}(\widehat{E}^{'\sigma}(L)).
	\end{tikzcd}\] 
Combining this diagram with the information in \autoref{filtration2} we obtain the following lemma.

\begin{lem}\label{filtration3} Consider the decomposition $E_L(L)/p\simeq\overline{U}_L^{p^2+1}\oplus\overline{U}_L^{p+1}$ given in \eqref{decompose}. Let $\sigma$ be the generator of $G_0=\Gal(L/k)$. Then 
\begin{enumerate}
\item For each $j\in\{1,\ldots,p-1\}$, $\sigma$ induces an automorphism $\displaystyle\frac{\overline{U}_L^{p^2+j}}{\overline{U}_L^{p^2+j+1}}\xrightarrow{\overline{\sigma}}\frac{\overline{U}_L^{p^2+j}}{\overline{U}_L^{p^2+j+1}}$ given by multiplication by $\overline{u}^{pj}$. 
\item For each $i\in\{1,\ldots,p^2-1\}$  
which is coprime to $p$, 
$\sigma$ induces an automorphism $\displaystyle\frac{\overline{U}_L^{p+i}}{\overline{U}_L^{p+i+1}}\xrightarrow{\overline{\sigma}}\frac{\overline{U}_L^{p+i}}{\overline{U}_L^{p+i+1}}$ given by multiplication by $\overline{u}^{i}$. 
\end{enumerate}
\end{lem}
\begin{proof}
This follows directly by \autoref{filtration2}. Namely, to prove (1) note that, for each $j\in\{1,\ldots,p-1\}$ which is coprime to $p$, we have a commutative diagram of  isomorphisms,
\[\begin{tikzcd}
	\overline{U}_L^{p^2+j}/\overline{U}_L^{p^2+j+1}\ar{d}{\overline{\sigma}}\ar{r}{\simeq} & \widehat{E}^{'j}(L)/\widehat{E}^{'j+1}(L)\ar{r}{\check{\phi}}\ar{d}{\overline{\sigma}}& \widehat{E}_L^{pj}(L)/\widehat{E}_L^{pj+1}(L)\ar{d}{\overline{\sigma}} \\
\overline{U}_L^{p^2+j}/\overline{U}_L^{p^2+j+1}\ar{r}{\simeq} & (\widehat{E}^{'\sigma})^{ j}(L)/(\widehat{E}^{'\sigma})^{ j+1}(L)\ar{r}{\check{\phi^\sigma}} & \widehat{E}_L^{pj}(L)/\widehat{E}_L^{pj+1}(L),
	\end{tikzcd}\] where the leftmost vertical map is obtained by the diagram. Since the rightmost vertical map is given by multiplication by $\overline{u}^{pj}$, the claim follows. 

Similarly, to prove (2) note that for each $i\in\{1,\ldots,p^2-1\}$ we have a commutative diagram of isomorphisms, 

 \[\begin{tikzcd}
	\widehat{E}_L^{i}(L)/\widehat{E}_L^{i+1}(L)\ar{d}{\overline{\sigma}}\ar{r}{\varepsilon} &  \mathcal{F}_L^{i}/\mathcal{F}_L^{i+1} \ar{r}{\simeq}\ar{d}{\overline{\sigma}}& \ar{d}{\overline{\sigma}} \overline{U}_L^{p+i}/\overline{U}_L^{p+i+1}\\ 
	\widehat{E}_L^{i}(L)/\widehat{E}_L^{i+1}(L)
\ar{r}{\varepsilon^\sigma} & (\mathcal{F}^{\sigma})_L^i/(\mathcal{F}^{\sigma})_L^{i+1}
\ar{r}{\simeq} 
& \overline{U}_L^{p+i}/\overline{U}_L^{p+i+1} , 
	\end{tikzcd}\] 
where 
$\mathcal{F}_L^i = \widehat{E}_L^{i}(L)/(\widehat{E}_L^{i}(L)\cap\check{\phi}(\widehat{E}^{'}(L)))$, 
$(\mathcal{F}^{\sigma})_L^i = \widehat{E}_L^{i}(L)/(\widehat{E}_L^{i}(L)\cap\check{\phi}^{\sigma}(\widehat{E}^{'\sigma}(L)))$, and 
	the rightmost vertical map is obtained by the diagram. Since the leftmost vertical map is given by multiplication by $\overline{u}^{i}$, the claim follows. 
%
%
\end{proof}

\begin{rem} 
We note that the various isomorphisms in Lemmas \ref{sigmaaction} and  \ref{filtration3} are not canonical; namely they depend on the choice of a uniformizer element $\pi$ of $L$. 
\end{rem}

\smallskip 
\section{Main Results}
%
In this section we give proofs for theorems \ref{maintheointro} and \ref{main3intro}. All our statements will be in terms of the Mackey product $(E_1\otimes E_2)(k)$, and the Somekawa $K$-group $K(k;E_1,E_2)$. Using the identification of the latter with the Albanese kernel $F^2(E_1\times E_2)$ (cf.\ \eqref{zerocycles}), the corresponding statements for zero-cycles will follow. We start with some preliminary lemmas. 

\begin{lem}\label{formalproduct} Let $E_1, E_2$ be elliptic curves with good reduction over a finite extension $k$ of $\Q_p$. Then, we have a surjection of Mackey functors,
\[\frac{\widehat{E}_1\otimes\widehat{E}_2}{p}\rightarrow\frac{E_1\otimes E_2}{p}\rightarrow 0.\]
\end{lem}
\begin{proof} We have an exact sequence of Mackey functors,
\[\widehat{E}_1/p\rightarrow E_1/p\rightarrow [E_1/\widehat{E}_1]/p\rightarrow 0.\] Applying the right exact functor $\otimes E_2/p$, we obtain an exact sequence of Mackey functors,
\[\widehat{E}_1/p\otimes E_2/p\rightarrow E_1/p\otimes E_2/p\rightarrow [E_1/\widehat{E}_1]/p\otimes E_2/p\rightarrow 0.\] We claim that $[E_1/\widehat{E}_1]/p\otimes E_2/p=0$, that is $([E_1/\widehat{E}_1]/p\otimes E_2/p)(K)=0$, for any finite extension $K/k$. To see this, let $\{x,y\}_{F/K}\in([E_1/\widehat{E}_1]/p\otimes E_2/p)(K)$, where $F/K$ is some finite extension, $x\in [E_1/\widehat{E}_1](F)/p$ and $y\in E_2(F)/p$. Recall that $[E_1/\widehat{E}_1](F)=\overline{E}_1(\F_F)$. Let $x'\in \overline{E}_1(\overline{\F_F})$ be a point such that $x=px'$. Then $x'$ is defined over some finite extension $\F'/\F_F$. 
Let $F'/F$ be a finite unramified extension with 
	$\F_{F'} = \F'$ (cf.\ ~\cite[I.6]{serre1979local}). 
Since the elliptic curve $E_2$ has good reduction,  \cite[Cororally 4.4]{Mazur1972} gives that the norm map $E_2(F')\xrightarrow{N_{F'/F}}E_2(F)$ is surjective. Thus, we can find   $y'\in E_2(F')$ such that $y=N_{F'/F}(y)$. The projection formula \eqref{projectionformula} yields (cf.\ \autoref{vanishingtrick}), 
\[\{x,y\}_{F/K}=\{x,N_{F'/F}(y)\}_{F/K}=N_{F'/F}(\{x,y'\}_{F'/F'})=N_{F'/F}(\{px',y'\}_{F'/F'})=0.\] We conclude that there is a surjection of Mackey functors,
\[\widehat{E}_1/p\otimes E_2/p\rightarrow E_1/p\otimes E_2/p\rightarrow 0.\] 
With a similar argument, by applying the right exact functor $\widehat{E}_1/p\otimes$ to the exact sequence \eqref{ses-2} for $E_2$ we obtain an exact sequence of Mackey functors,
\[\widehat{E}_1/p\otimes \widehat{E}_2/p\rightarrow \widehat{E}_1/p\otimes E_2/p\rightarrow \widehat{E}_1/p\otimes[E_2/\widehat{E}_2]/p\rightarrow 0.\] We claim that $\widehat{E}_1/p\otimes[E_2/\widehat{E}_2]/p=0$. The argument is exactly the same as before, noting that if $F'/F$ is a finite unramified extension, then the norm map on formal groups, $\widehat{E}_1(\mathfrak{m}_{F'})\xrightarrow{N_{F'/F}}\widehat{E}_1(\mathfrak{m}_{F})$ is surjective (cf.\ \cite[Proposition 3.1]{Hazewinkel1974}). 

\end{proof}

\begin{notn}\label{Symb} 
For $X = E_1\times E_2$ we will denote by $\Symb_X(k)$ the subgroup of 
$(E_1\otimes E_2)(k)$   
generated by symbols of the form $\{x,y\}_{k/k}$. 
Moreover, let $\overline{\Symb}_X(k)$ be the image of $\Symb_X(k)$ 
 in  $(E_1\otimes E_2)(k)/p$.  
\end{notn}

\begin{lem}\label{extendunramf} Let $E_1, E_2$ be elliptic curves over a finite  extension $k$ of $\Q_p$. Assume that at least one of the curves has good reduction. Let $k'/k$ be a  finite unramified extension. Then,
\begin{enumerate}
\item We have an inclusion $\overline{\Symb}_X (k)\subseteq N_{k'/k}(\overline{\Symb}_X (k'))$ in  $(E_1\otimes E_2)(k)/p$.
\item If the group  $(E_1\otimes E_2)(k')$   is $p$-divisible, then so is 
 $(E_1\otimes E_2)(k)$.  
\end{enumerate}
\end{lem}
\begin{proof} Without loss of generality, assume that $E_1$ has good reduction. To prove (1), let $\{a,b\}_{k/k}\in\overline{\Symb}_X(k)$. Because $k'/k$ is unramified, and $E_1$ has good reduction, the norm map $E_1(k')\xrightarrow{N_{k'/k}} E_1(k)$ is surjective (\cite[Cororally 4.4]{Mazur1972}).
Thus, we may find $a'\in E_1(k')$ such that $a=N_{k'/k}(a')$. This yields equalities,
\[\{a,b\}_{k/k}=\{N_{k'/k}(a'),b\}_{k/k}=\{a',b\}_{k'/k}=N_{k'/k}(\{a',b\}_{k'/k'}).\] 

To prove (2), assume that  $(E_1\otimes E_2)(k')$  is $p$-divisible. Let  $\{a,b\}_{F/k}\in (E_1\otimes E_2)(k)/p$,  where $F/k$ is some finite extension. Consider the extension $F'=F\cdot k'$. Then $F'/F$ is unramified, and hence the norm $E_1(F')\xrightarrow{N_{F'/F}}E_1(F)$  is surjective. Thus there exists some $a'\in E_1(F')$ such that $a=N_{F'/F}(a')$. The projection formula yields (cf.\ \autoref{vanishingtrick}),
\[
 \{a,b\}_{F/k}= \{N_{F'/F}(a'),b\}_{F/k}=\{a',b\}_{F'/k}=N_{k'/k}(\{a',b\}_{F'/k'})=0\in  (E_1\otimes E_2)(k)/p,
 \] and hence  $(E_1\otimes E_2)(k)/p=0$. 

\end{proof}

\subsection{Proof of \autoref{maintheointro}} 
As already mentioned in the introduction, the proof of this theorem  will be split up into two steps. 
Let $k$ be a finite unramified extension of $\Q_p$ 
and let $E_1, E_2$ be elliptic curves over $k$ with good reduction. 
Put $X = E_1\times E_2$. 
Suppose that  $E_1$ has good ordinary reduction. 
In the following, we consider the following two cases: 

\noindent 
(\textbf{ord}) $E_2$ has good ordinary reduction,  

\noindent
(\textbf{ssing}) $E_2$ has good supersingular reduction. 

\smallskip
We can consider the smallest  extension $L/k$  such that the following are true:
\begin{itemize}
\item  $L\supset k(\widehat{E}_1[p],\widehat{E}_2[p])$, 
\item If the curve $E_i$ has good ordinary reduction, then the $G_L$-module $E_{iL}[p]$ fits into a short exact sequence of the form \eqref{ses2}, and hence $\overline{E}_i[p]\subset \overline{E}_i(\F_L)$. 
\end{itemize}
The above assumptions imply that $L\supset k(\mu_p)$ and that $L$ fits into a tower
 of finite extensions 
 \begin{equation}\label{extensionL}
	k\subset k'\subset L 	
 \end{equation}
 with $k'/k$ unramified and $L/k'$ totally ramified of degree coprime to $p$. More precisely, sections \ref{ordinarysection}-\ref{unramified section}  imply the following, 
\begin{itemize}
\item 
In the case (\textbf{ord}), then $L/k$ has ramification index $p-1$, and hence $e_0(L)=1$. 
\item In the case (\textbf{ssing}), 
then $L/k$ has ramification index $p^2-1$, 
and hence $e_0(L)=p+1$. 
\end{itemize}

\begin{theo}\label{main2} Let $L/k'/k$ be the tower of extensions as in \eqref{extensionL}. Then, we have an equality $\overline{\Symb}_{X_{k'}}(k')=  (E_1\otimes E_2)(k')/p$.  
\end{theo}

\begin{proof} For simplicity we assume that the extension $L/k$ is totally ramified, and hence $k=k'$. 
We need to show that the group  $(E_1\otimes E_2)(k)/p$  is generated by symbols of the form $\{a,b\}_{k/k}$. 
Since $[L:k]$ is coprime to $p$, the norm map 
$(E_1\otimes E_2)(L)/p \xrightarrow{N_{L/k}} (E_1\otimes E_2)(k)/p$  
 is surjective.
It is enough therefore to show that the image of the norm lies in $\overline{\Symb}_X(k)$.

\smallskip
\noindent
\textit{Claim 1:}  $\overline{\Symb}_{X_L}(L) =(E_1\otimes E_2)(L)/p$.   That is, the group  $(E_1\otimes E_2)(L)/p$  is generated by symbols of the form $\{x,y\}_{L/L}$, with $x\in E_1(L)$, $y\in E_2(L)$. 

This claim follows easily by the computations in \cite{Hiranouchi2014} 
as follows:
Using \autoref{formalproduct}, 
the problem is reduced to showing that the Mackey product $(\widehat{E}_1/p\otimes \widehat{E}_2/p)(L)$ is generated by symbols of the form  $\{x,y\}_{L/L}$, with $x\in \widehat{E}_1(L)$, $y\in \widehat{E}_2(L)$. We have the following subcases. 

\begin{itemize}
\item In the case (\textbf{ord}),  \autoref{formalgroup} gives for $i=1,2$ an isomorphism 
$\widehat{E}_i/p\simeq \overline{U}^0$ of Mackey functors over $L$. This implies an isomorphism of abelian groups, 
\[(\widehat{E}_1/p\otimes \widehat{E}_2/p)(L)\simeq(\overline{U}^0\otimes\overline{U}^0)(L).\] 
\item In the case (\textbf{ssing}), \eqref{decompose} gives an isomorphism 
$\widehat{E}_2/p\simeq \overline{U}^{p^2}\oplus \overline{U}^{p}$ of Mackey functors over $L$,
and \cite[Lemma 3.3]{Hiranouchi2014} gives,  
\[(\widehat{E}_1/p\otimes \widehat{E}_2/p)(L)\simeq(\overline{U}^0\otimes\overline{U}^{p^2})(L)\oplus(\overline{U}^0\otimes\overline{U}^{p})(L).\] 
\end{itemize}
The argument discussed in 
Proposition 3.11 of \cite{Hiranouchi2014} gives that these Mackey products $(\overline{U}^0\otimes \overline{U}^j)(L)$ are generated by symbols defined over $L$. Thus, Claim 1 holds. 



\smallskip
By \cite[Theorem 3.6]{Hiranouchi2014}, 
for each $j< pe_0(L)$, 
the composition 
\begin{equation}
\label{product}	
(\overline{U}^0\otimes \overline{U}^j)(L) \xrightarrow{\iota}  
(\mathbb{G}_m\otimes\mathbb{G}_m)(L)/p \xrightarrow{s_p}
	  H^2(L,\mu_p^{\otimes 2}) \simeq \Z/p,
\end{equation}
is bijective, 
where 
$\iota$ is induced by the maps  $\overline{U}^i \hookrightarrow \mathbb{G}_m/p$, and 
$s_p$ is the Galois symbol map (cf.\ \eqref{classical symbol}). 
By \eqref{Tate}, for a symbol $\{x,y\}_{L/L}$ in $(\overline{U}^0\otimes \overline{U}^j)(L)$
\begin{equation}
	\label{generator}
\mbox{$\{x,y\}_{L/L}$ generates $(\overline{U}^0\otimes \overline{U}^j)(L)$} \Leftrightarrow  x \not\in N_{L_1/L}(L_1^{\times}),
\end{equation}
where 
$L_1 = L(\sqrt[p]{y})$.

\smallskip
\noindent
\textit{Claim 2:} Suppose we are in the situation (\textbf{ord}). Then,  the Mackey product $(E_1\otimes E_2)(L)/p$  is generated by symbols of the form $\{x,y\}_{L/L}$ with either $x\in E_1(k)$, or $y\in E_2(k)$. 

It is enough to show that the group $(\widehat{E}_1/p\otimes \widehat{E}_2/p)(L)\simeq(\overline{U}^0\otimes\overline{U}^{0})(L)$ is generated by a symbol $\{x,y\}_{L/L}$ with $x\in \widehat{E}_1(k)$. Take any $x\in \widehat{E}_1^1(k)\setminus \widehat{E}_1^2(k)$. Then its restriction in $\widehat{E}_1(L)$ lies in $\widehat{E}_1^{p-1}(L)\setminus\widehat{E}_1^{p}(L)$. In particular, we can view $x$ as an element of $\overline{U}_L^{p-1}\setminus\overline{U}_L^{p}$. 
The extension $L_1 = L(\sqrt[p]{x})/L$ is totally ramified extension of degree $p$ with the jump at $s = pe_0(L) - (p-1) = 1$ (cf.\ \autoref{jump} ($a$)). 
Then the isomorphism $U_L/N_{L_1/L}(U_{L_1}) \simeq L^{\times}/N_{L_1}(L_1^{\times})$ of order $p$
(cf.\ \cite[V.3, Corollary 7]{serre1979local})  gives us that there exists some 
$y\in U_L$ such that $y \not \in N_{L_1/L}(U_{L_1})$. By \eqref{generator}, the symbol $\{x,y\}_{L/L}$ generates the group $(\overline{U}^0 \otimes \overline{U}^0)(L) \simeq (\widehat{E}_1/p\otimes \widehat{E}_2/p)(L)$.  
 
The above computation guarantees that in the case (\textbf{ord}), 
$(E_1\otimes E_2)(k)/p$ 
coincides with $\overline{\Symb}_X(k)$.
For, find elements $\{x_i,y_i\}_{L/L}\in  (E_1\otimes E_2)(L)/p$  that generate this group and are such that $x_i\in E_1(k)$. Then $N_{L/k}(\{x_i,y_i\}_{L/L})$ generate  $(E_1\otimes E_2)(L)/p$.  But the projection formula yields,
\[N_{L/k}(\{x_i,y_i\}_{L/L})=\{x_i,N_{L/k}(y_i)\}_{k/k}\in\overline{\Symb}_X(k).\] 
This completes the proof in the case (\textbf{ord}). 

Lastly, we consider the case (\textbf{ssing}), namely, $E_1$ has good ordinary and $E_2$ has good supersingular reduction (hence $e_0(L)=p+1$). In this case we have a commutative diagram with exact rows and surjective vertical maps,
\[\begin{tikzcd}
	(\overline{U}^0\otimes\overline{U}^{p^2})(L)\oplus(\overline{U}^0\otimes\overline{U}^{p})(L)\ar{r} \ar{d}{N_{L/k}} &  (E_1\otimes E_2)(L)/p \ar{r}\ar{d}{N_{L/k}} & 0\\
	(\widehat{E}_1/p\otimes \widehat{E}_2/p)(k)\ar{r}  &  (E_1\otimes E_2)(k)/p \ar{r} & 0.
	\end{tikzcd}\] 
	 With a similar argument as in case (\textbf{ord}) we can show that the subgroup $(\overline{U}^0\otimes\overline{U}^{p})(L)$ of $(\widehat{E}_1/p\otimes \widehat{E}_2/p)(L)$ is generated by a symbol $\{x,y\}_{L/L}$ with $y\in \widehat{E}_2^1(k)$. 
	 Namely, let $y\in \widehat{E}_2^1(k)\setminus\widehat{E}_2^2(k)$.  Then by \autoref{resdec} (1) the image of $y$ in $E_2(L)/p$ decomposes as $(y_0,y_1)$, with $y_0\in\overline{U}_L^{p^2+p}$, and $y_1\in \overline{U}_L^{p+p^2-1}\setminus\overline{U}^{p+p^2}_L$. Since the extension $L(\sqrt[p]{y_1})/L$ is totally ramified of degree $p$ (cf.\ \autoref{jump} ($a$)), by \cite[V.3, Corollary 7]{serre1979local} we can find some $x\in\overline{U}_L^0$ such that $\{x,y_1\}_{L/L}\neq 0\in (\overline{U}^0\otimes\overline{U}^{p})(L)$. Viewing $x$ as an element of $E_1(L)/p$, we have that $\{x,y\}_{L/L}$ generates the subgroup $(\overline{U}^0\otimes\overline{U}^{p})(L)\subset (\widehat{E}_1/p\otimes \widehat{E}_2/p)(L)$. 
	 Notice that such an argument won't work for the other subgroup of $(\widehat{E}_1/p\otimes \widehat{E}_2/p)(L)$. Namely, no matter which symbol $\{x,y\}_{L/L}$ we take with either $x\in \widehat{E}_1(k)/p$ or $y\in \widehat{E}_2(k)/p$, the coordinate $\{x,(y_0,1)\}_{L/L}$ will always vanish. For, if $y\in\widehat{E}_2(k)/p$ this follows immediately from \autoref{resdec}. On the other hand, if $x\in\widehat{E}_1(k)/p$, then $x\in \widehat{E}_1^{p^2-1}(L)/p$. An easy computation shows that for $y=(y_0,y_1)\in \widehat{E}_2(L)/p$, and $L'=L(\sqrt[p]{y_0})$, it follows that $x\in N_{L'/L}(L^{'\times})$. We suggest the following claim instead. 

\smallskip
\noindent
\textit{Claim 3:} The norm map $(\overline{U}^0\otimes\overline{U}^{p^2})(L)\xrightarrow{N_{L/k}}(\widehat{E}_1/p\otimes \widehat{E}_2/p)(k)$ is zero. 

To prove this claim let $\{a,b\}_{L/L}$ be a generator of  $(\overline{U}^0\otimes\overline{U}^{p^2})(L)\simeq\Z/p$. For notational simplicity we will identify this symbol with its image, $\{a,(b,1)\}_{L/L}\in(\widehat{E}_1/p\otimes \widehat{E}_2/p)(L)$. We will show $N_{L/k}(\{a,b\}_{L/L})=0$. Since the restriction  $\res_{L/k}$ is injective, it suffices to show \[\res_{L/k}(N_{L/k}(\{a,b\}_{L/L}))=0\in(\overline{U}^0\otimes\overline{U}^{p^2})(L).\] We first obtain some information about the generator $\{a,b\}_{L/L}$. 

As noted in \eqref{generator}, the symbol $\{a,b\}_{L/L}$ is a generator of  $(\overline{U}^0\otimes\overline{U}^{p^2})(L)$ if and only if $a\not\in N_{L_1/L}(L_1^\times)$, where $L_1=L(\sqrt[p]{b})$.
 Suppose that $a\in\overline{U}_L^i\setminus\overline{U}_L^{i+1}$, and $b\in\overline{U}_L^{p^2+j}\setminus\overline{U}_L^{p^2+j+1}$. 
By \cite[Lemma 3.4 (ii)]{Hiranouchi2014} we may assume that $i$ is coprime to $p$, $i<p^2+p$ and $1\leq j<p$, otherwise $\{a,b\}_{L/L}\neq 0$, and hence $\{a,b\}_{L/L}$ is not a generator of $(\overline{U}^0\otimes\overline{U}^{p^2})(L)$. Moreover, by \cite[V.3, Corollary 7]{serre1979local} we obtain the following equivalence 
\[\{a,b\}_p\neq 0\Leftrightarrow i+p^2+j=pe_0(L)=p^2+p\Leftrightarrow i+j=p.\]
 Let $G_0=\Gal(L/k)$. Note that $L=k(\widehat{E}_2[p])$, and hence $G_0$ is a cyclic group of order $p^2-1$ (see Section~\ref{unramified section} for details). Let $\sigma$ be a generator of $G_0$. Then \autoref{restriction} yields, 
\[\res_{L/k}(N_{L/k}(\{a,b\}_{L/L}))=\sum_{r=0}^{p^2-2}\sigma^r(\{a,b\}_{L/L}).\] 
Set $V=(\overline{U}^0\otimes\overline{U}^{p^2})(L)$. 
Using the previous remarks, the symbol $\{a,b\}_{L/L}$ is in the image of 
 the symbol map $\overline{U}_L^i \otimes\overline{U}_L^{p^2+j}\rightarrow V$ defined by $x\otimes y \mapsto \{x,y\}_{L/L}$. 
 This map factors through $\overline{U}_L^i/\overline{U}_L^{i+1}\otimes\overline{U}_L^{p^2+j}/\overline{U}_L^{p^2+j+1}$, inducing a symbol map
\begin{equation}\label{symbolmap}
\overline{U}_L^i/\overline{U}_L^{i+1}\otimes\overline{U}_L^{p^2+j}/\overline{U}_L^{p^2+j+1}\rightarrow V.
\end{equation}
 In fact, for any $x \in \overline{U}_L^l, y \in \overline{U}_L^{p^2+l'}$ with $l+l'= p+1$, 
	$s_p(\{x,y\}_{L/L}) = (x,y)_p = 0$ by \cite[Lemma 3.4]{Hiranouchi/Hirayama2013} as $l+p^2+l' = p^2 + p+1 > pe_0(L)$. This implies $\{x,y\}_{L/L} = 0$.
Let $\overline{\sigma}$ be the endomorphism of 
$\overline{U}_L^i/\overline{U}_L^{i+1}\otimes\overline{U}_L^{p^2+j}/\overline{U}_L^{p^2+j+1}$ induced by $\sigma$.
Fix a uniformizer $\pi_L$ of $L$. Then \autoref{filtration3}  yields the following equality in $\overline{U}_L^i/\overline{U}_L^{i+1}\otimes\overline{U}_L^{p^2+j}/\overline{U}_L^{p^2+j+1}$,
\[\overline{\sigma}(\overline{a}\otimes\overline{b}) 
=\overline{\sigma}(\overline{a})\otimes \overline{\sigma}(\overline{b})
=\left(\frac{\sigma(\pi_L)}{\pi_L}\right)^{i}\overline{a}\otimes \left(\frac{\sigma(\pi_L)}{\pi_L}\right)^{pj}\overline{b} 
=\overline{u}^i\overline{a}\otimes \overline{u}^{pj}\overline{b} ,
\] 
where $\overline{u}=\sigma(\pi_L)/\pi_L\in\F_L^\times$. Note that since $\sigma$ has exact order $p^2-1$, $\overline{u}$ lies in the (possibly smaller) field $\F_{p^2}$. 
The next claim is that we have an equality, $\overline{u}^i\overline{a}\otimes \overline{u}^{pj}\overline{b}=\overline{u}^{i+pj}\overline{a}\otimes \overline{b}$. This follows by \autoref{units}. Namely, we have isomorphisms $\overline{U}_L^i/\overline{U}_L^{i+1}\simeq\F_L$, and $\overline{U}_L^{p^2+j}/\overline{U}_L^{p^2+j+1}\simeq\F_L$. Thus, the group $\overline{U}_L^i/\overline{U}_L^{i+1}\otimes\overline{U}_L^{p^2+j}/\overline{U}_L^{p^2+j+1}$ becomes an $\F_L$-vector space by defining for $c\in\F_L$ and $\overline{x}\otimes \overline{y} \in \overline{U}_L^i/\overline{U}_L^{i+1}\otimes\overline{U}_L^{p^2+j}/\overline{U}_L^{p^2+j+1}$,
\[
c\cdot\overline{x}\otimes \overline{y}
= (c\overline{x})\otimes \overline{y}
=\overline{x}\otimes (c\overline{y}).
\]  
We conclude that $\overline{\sigma}$ acts as follows,
\[
\overline{\sigma}(\overline{a})\otimes \overline{\sigma}(\overline{b}) 
= \overline{u}^{i+pj}\overline{a}\otimes \overline{b}.
\]  Since $i+j=p$, and both $i,j$ are coprime to $p$, we have $1\leq j\leq p-1$, and hence $pj\leq p^2-p$. This gives  \[i+pj\leq p^2-p+p-1=p^2-1.\] In order for the inequality to become an equality we need both $j=p-1$ and $i=p-1$. But this is not true, since $i+j=p$. Thus, $i+pj<p^2-1$. Set $\overline{v}=\overline{u}^{i+pj}\in\F_{p^2}$. The above inequality implies that $\overline{v}\neq 1$. At the same time $\overline{v}^{p^2-1}=1$, which means that $\overline{v}$ is a root of the polynomial $f(x)=x^{p^2-2}+x^{p^2-3}+\cdots+x+1\in\F_{p^2}[x]$. We then can compute,
\[
\sum_{r=0}^{p^2-2}\overline{\sigma}^r(\overline{a}\otimes \overline{b})
=\left(\sum_{r=0}^{p^2-2}\overline{v}^r\right)\overline{a}\otimes\overline{b}
=0 \quad \mbox{in}\ \overline{U}_L^i/\overline{U}_L^{i+1}\otimes\overline{U}_L^{p^2+j}/\overline{U}_L^{p^2+j+1}.
\] 
Then \eqref{symbolmap} yields a vanishing $\res_{L/k}(N_{L/k}(\{a,b\}_{L/L})=0$, which completes the proof. 

\end{proof}
The next theorem completes the proof of \autoref{maintheointro}. 

\begin{theo}\label{main1} 
Let  $k$ be a finite unramified extension of $\Q_p$. Let $E_1, E_2$ be elliptic curves over $k$  with good reduction. 
Assume that $E_1$ has good ordinary reduction.   
Then, the Mackey product $(E_1\otimes E_2)(k)$ is $p$-divisible. 
In particular, the same holds for $K(k;E_1,E_2)$. 
\end{theo}

\begin{proof}  
Using \autoref{extendunramf} (2), we may assume $k = k'$, 
where $k'/k$ is as in \eqref{extensionL}. 
By \autoref{main2}, it is enough to show 
 the vanishing of $\overline{\Symb}_X(k)$. 
 Notice that throughout the proof of \autoref{main2}, we used the Mackey product $(\widehat{E}_1\otimes\widehat{E}_2)(k)/p$. It is therefore enough to show that $\{a,b\}_{k/k}=0$, for every $a\in\widehat{E}_1(k)/p$ and $b\in\widehat{E}_2(k)/p$. Consider the finite extension $L/k$ as in \eqref{extensionL}, which is totally ramified of degree $p^2-1$ or $p-1$, depending on whether the curve $E_2$ has good supersingular reduction or not. 

\smallskip
\noindent
\textit{Case} (\textbf{ord}):
\textit{When both elliptic curves have good ordinary reduction.}
In this case we have $e_0(L)=1$. We consider the restriction map, $(\widehat{E}_1\otimes \widehat{E}_2)(k)/p\xrightarrow{\res_{L/k}}(\widehat{E}_1\otimes \widehat{E}_2)(L)/p$. Since the extension $L/k$ is of degree coprime to $p$, $\res_{L/k}$ is injective.  It is therefore enough to show that $\res_{L/k}(\{a,b\}_{k/k})=\{a,b\}_{L/L}=0$, for every $a\in\widehat{E}_1(k)/p$ and $b\in\widehat{E}_2(k)/p$.
Using the identification $\widehat{E}_i(L)/p\simeq \overline{U}_L^0$ for $i=1,2$ (\autoref{formalgroup}), 
we will show $\{a,b\}_{L/L} = 0$ in $(\overline{U}^0\otimes \overline{U}^0)(L)$.
Using the vanishing trick \eqref{vanishingtrick}, it is enough to show that 
$b\in\img\left(\overline{U}^0_{L_1}\xrightarrow{N_{L_1/L}} \overline{U}^0_L\right)$ for $L_1 = L(\sqrt[p]{a})$.
Since $a\in\widehat{E}_1^1(k)$, its restriction in $\widehat{E}_1(L)$ lies in $\widehat{E}^{p-1}_1(L)$, and the same holds for $b$. Thus in order to calculate the symbol $\{a,b\}_{L/L}$, we may view $a,b$ as units in $\overline{U}_L^{p-1}\subset\overline{U}_L^0$. We distinguish the following two cases. First, if $v_k(a)=1$, that is, $a\in\widehat{E}^1_1(k)\setminus\widehat{E}_1^2(k)$, then its restriction over $L$ lies in $\overline{U}_L^{p-1}\setminus\overline{U}_L^p.$ In this case \autoref{jump} ($a$) gives us that $L_1/L$ is a totally ramified degree $p$ extension and its Galois group $\Gal(L_1/L)$ has jump in its ramification filtration at $s=pe_0(L)-p+1=1$. Since $b\in\overline{U}_L^{p-1}$, and $p-1>s$, \autoref{psi} (2) gives a surjection $\overline{U}_{L_1}^{\psi(p-1)}\twoheadrightarrow\overline{U}_L^{p-1}$, 
where $\psi := \psi_{L_1/L}$. Since $\overline{U}_{L_1}^{\psi(p-1)}\subset\overline{U}_{L_1}^0$ and $\widehat{E}_2(L_1)/p\simeq\overline{U}_{L_1}^0$, we conclude that $b$ is in the image of the norm map, and hence  $\{a,b\}_{L/L}=0$. Second, if $v_k(a)>1$, then its restriction over $L$ lies in $\overline{U}_L^p\subset\overline{U}_L^0$. In this case the extension $L_1/L$ is unramified (cf.\ \autoref{jump} ($b$)), and hence the norm map 
$U_{L_1}\xrightarrow{N_{L_1/L}}U_L$ 
is surjective. This completes the proof for the case of two elliptic curves with good ordinary reduction. 

\smallskip
\noindent
\textit{Case} (\textbf{ssing}): \textit{When  $E_2$ has good supersingular reduction.} 
In this case we have $[L:k]=p^2-1$, and $e_0(L)=p+1$. 
The argument is analogous to the previous case. Namely, let $\{a,b\}_{k/k}\in(\widehat{E}_1\otimes \widehat{E}_2)(k)/p$. Since the extension $L/k$ is of degree coprime to $p$, in order to show that $\{a,b\}_{k/k}=0$, it is enough to show that $\{a,b\}_{L/L}=0$. Since $b\in\widehat{E}_2(k)$, $v_L(b)\geq p^2-1$. Let $F=L\left(\frac{1}{p}b\right)$. We will show that $a\in N_{F/L}(\widehat{E}_1(F)/p)$. \autoref{resdec} allows us to reduce to the case when $b\in\widehat{E}^{p^2-1}_2(L)\setminus\widehat{E}^{p^2}_2(L)$. Then \autoref{resdec} (1) gives us that $b$ decomposes as $(b_0,b_1)$ under the decomposition \eqref{decompose}, the extension $F_0=L(\sqrt[p]{b_0})$ is unramified over $L$, while the extension $F_1=L(\sqrt[p]{b_1})/L$ is totally ramified of degree $p$ and the jump in the ramification filtration of $\Gal(L(\sqrt[p]{b_1})/L)$ happens at $s=1$. We may write,
 \[\{a,b\}_{L/L}=\{a,(b_0,1)\}_{L/L}+\{a,(1,b_1)\}_{L/L}.\] We claim that both these symbols are zero. The first one vanishes, since the norm map $E_1(F_0)\xrightarrow{N_{F_0/L}} E_1(L)$ is surjective. For the second symbol, we can identify the restriction of $a$ over $L$ with a unit $a\in\overline{U}_L^{p^2-1}$. Since $p^2-1>s=1$, \autoref{psi} gives us that $a\in N_{F_1/L}(F_1^\times)$, and hence $a\in N_{F_1/L}(\widehat{E}_1(F_1)/p)$. 

\end{proof}

\subsection{The case of two elliptic curves with good supersingular reduction}\label{sscase} In this subsection we focus on the only case which is not covered in \autoref{maintheointro}; namely when $X=E_1\times E_2$ is a product of two elliptic curves with good supersingular reduction. This case appears to be much harder than all others, and  we can only obtain a partial result. Our first computation shows that a weaker form of \autoref{main1} is still true in this case.

\begin{theo}\label{ssing1} Let $X=E_1\times E_2$ be a product of two elliptic curves with good supersingular reduction over an unramified extension $k$ of $\Q_p$. Then, we have $\overline{\Symb}_X(k)=0$. 
\end{theo}
\begin{proof} The proof is very analogous to \autoref{main1}. 
Consider the finite extension $L=k(E_1[p], E_2[p])$. 
Using \autoref{extendunramf} we may assume that $L/k$ is totally ramified. 
Since for $i=1,2$ the extension $k_i := k(E_i[p])/k$ is totally ramified of degree $p^2-1$, 
 $L/k$ is a tamely ramified extension of degree $e(p^2-1)$, 
 where $e = [L:k_1] = [L:k_2]$. 
 
It follows by \eqref{ssing0}, that once again it suffices to show the vanishing of every symbol  $\{a,b\}_{k/k}\in(\widehat{E}_1\otimes \widehat{E}_2)(k)/p$.  
 Let $\{a,b\}_{k/k}\in(\widehat{E}_1\otimes \widehat{E}_2)(k)/p$. 
 We will show that $\{a,b\}_{L/L}=0$. 
 Decompose 
 $a = (a_0,a_1) \in \widehat{E}_1(k_1)/p$, and  $b = (b_0,b_1) \in \widehat{E}_2(k_2)/p$ by identifying   
\[
\widehat{E}_i(k_i)/p \simeq \overline{U}_{k_i}^{p(e_0(k_i) - t_0(k_i))} \oplus \overline{U}_{k_i}^{pt_0(k_i)} = \overline{U}_{k_i}^{p^2} \oplus \overline{U}_{k_i}^{p}.
\]
 From \autoref{resdec}, 
 we may assume that $v_k(a)= v_k(b) = 1$ and we have   
 \begin{align*} 	
 a_0\in\overline{U}_{k_1}^{p^2+p}\quad \mbox{and}\quad a_1\in\overline{U}_{k_1}^{p^2+p-1}\smallsetminus \overline{U}_{k_1}^{p^2+p}, \\
 b_0\in\overline{U}_{k_2}^{p^2+p}\quad \mbox{and}\quad b_1\in\overline{U}_{k_2}^{p^2+p-1}\smallsetminus \overline{U}_{k_2}^{p^2+p}. 
 \end{align*}
 By restricting them to $L$, we obtain 
 \[
 a=(a_0,a_1),\ b= (b_0,b_1) \quad \mbox{with $a_0,b_0 \in\overline{U}_{L}^{e(p^2+p)}$ and $a_1,b_1 \in\overline{U}_{L}^{e(p^2+p-1)}\smallsetminus \overline{U}_{L}^{e(p^2+p)}$}
 \]
 in the decomposition
 $
 \widehat{E}_i(L)/p \simeq \overline{U}_{L}^{p(e_0(L) - t_0(L))} \oplus  \overline{U}_{L}^{pt_0(L)} = \overline{U}_{L}^{ep^2} \oplus  \overline{U}_{L}^{ep}.
$
 This in turn gives,  
\[\{a,b\}_{L/L}=\{(a_0,1),(b_0,1)\}_{L/L}+\{(1,a_1),(b_0,1)\}_{L/L}+\{(a_0,1),(1,b_1)\}_{L/L}+\{(1,a_1),(1,b_1)\}_{L/L}.\] 
Since the extensions $L(\sqrt[p]{a_0})$, and $L(\sqrt[p]{b_0})$ are unramified over $L$, the first three symbols vanish. It remains to show $\{(1,a_1),(1,b_1)\}_{L/L}=0$. Consider the finite extension $L_1=L(\sqrt[p]{b_1})$. We need to show that $(1,a_1)\in N_{L_1/L}(\widehat{E}_1(L_1)/p)$. 
By the same reason, we may assume that $L_1/L$ is ramified. 
\autoref{jump} gives us that 
$L_1/L$ is totally ramified of degree $p$ with jump at $s= pe_0(L) - i = e(p^2+p) - i$ 
	for some $e(p^2+p-1) \le i < e (p^2+p)$. 
	In particular, we have $0 < s \le e$. 
	Considering the decomposition 
	\[
	\widehat{E}_1(L_1)/p\simeq\overline{U}_{L_1}^{ep^3}\oplus\overline{U}_{L_1}^{ep^2},
	\] 
	it is enough to show that $a_1\in\img(\overline{U}_{L_1}^{ep^2}\xrightarrow{N_{L_1/L}}\overline{U}_{L}^{ep})$. 
	By 
	$e(p^2+p-1) > e \ge s$, 
	 \autoref{psi} gives us a surjection \[\overline{U}_{L_1}^{\psi(e(p^2+p-1))}\xrightarrow{N_{L_1/L}}\overline{U}_L^{e(p^2+p-1)}\rightarrow 0.\] In order for  $(1,a_1)$ to be a norm, we need to verify that $\overline{U}_{L_1}^{\psi(e(p^2+p-1))}\subset\overline{U}_{L_1}^{ep^2}$. 
	 This follows from $\psi(e(p^2+p-1))\geq e(p^2+p-1) > ep^2$.
	 Here, the first inequality follows from the definition of the Herbrand function.

\end{proof}

In order to extend \autoref{maintheointro} to this case, we would need an analog of \autoref{main2}. 
The problem is that in this case we do not know how large the $K$-group $K(L;E_1,E_2)/p$ is. Let us focus for simplicity on the case of a self-product $X=E\times E$. 
It follows by \cite[Proposition~3.6]{Hiranouchi/Hirayama2013} that the image of the  Galois symbol $(E/p\otimes E/p)(L)\xrightarrow{s_p}H^2(L,E[p]\otimes E[p])$ is isomorphic to $\Z/p$. In all other reduction cases this map is an isomorphism (\cite{Raskind/Spiess2000, Yamazaki2005, Hiranouchi2014}). 
\begin{que}\label{injectivity} Is the map $(E/p\otimes E/p)(L)\xrightarrow{s_p}H^2(L,E[p]\otimes E[p])$ injective when $E$ has good supersingular reduction? 
\end{que} It is very likely that \autoref{injectivity} has a negative answer. At least the Mackey product $(E/p\otimes E/p)(L)\simeq(\overline{U}^{p}\otimes\overline{U}^{p^2})(L)$ does not seem to give enough relations that guarantee injectivity. 
A weaker question is whether the analog of Claim 1 in the proof of \autoref{main2} is true in this case. 
\begin{que}\label{Lsymbols} Is the $K$-group $K(L;E_1,E_2)/p$ generated by symbols of the form $\{a,b\}_{L/L}$, with $a\in E_1(L), b\in E_2(L)$? 
\end{que}


  The next theorem provides some indication that if \autoref{Lsymbols} has an affirmative answer, then  \autoref{main2} can indeed be extended to this case.

\begin{theo}\label{main3} Let $k$ be a finite unramified extension of $\Q_p$. Suppose $X=E\times E$ is the self-product of an elliptic curve over $k$ with good supersingular reduction. Let $L=k(E[p])$.
Then, for every $a, b\in E(L)/p$, we have  
\[N_{L/k}(\{a,b\}_{L/L})=\{a,b\}_{L/k}=0\]
in the Mackey product $(E\otimes E)(k)/p$.
\end{theo}
\begin{proof}
The proof will be along the lines of Claim 3 in the proof of \autoref{main2}.  With the usual argument we may assume that the extension  $L=k(\widehat{E}[p])/k$ is totally ramified of degree $p^2-1$ with cyclic Galois group. By \autoref{formalproduct} and \eqref{ssing0} it is enough to prove that for every $a,b\in\widehat{E}(L)/p$, it holds $N_{L/k}(\{a,b\}_{L/L})=0,$ and since  $\res_{L/k}$ is injective,
this is equivalent to proving that $\res_{L/k}(N_{L/k}(\{a,b\}_{L/L}))=0$.

Set $W=(\widehat{E}/p\otimes\widehat{E}/p)(L)$. Recall that the group $\widehat{E}(L)/p$ has a filtration $\{\mathcal{D}_L^i\}$ given by $\mathcal{D}_L^i=\widehat{E}^i(L)/(p\widehat{E}(L)\cap\widehat{E}^i(L))$, for $i\geq 1$. This induces a filtration on $W$ as follows.  For every $t\geq 1$ the symbol maps define 
\[
\Fil^t(W):= \sum_{n+m=t} \img(\mathcal{D}_L^n\otimes\mathcal{D}_L^m\to W).
\] 

\smallskip
\noindent
\textit{Claim 1:} Proving $\res_{L/k}\circ N_{L/k}=0$ is equivalent to proving that if $\{a,b\}_{L/L}\in\Fil^t(W)$ for some $t\geq 1$, then $\res_{L/k}(N_{L/k}(\{a,b\}_{L/L}))\in\Fil^{t+1}(W)$. 

The direction $(\Rightarrow)$ is clear. To prove $(\Leftarrow)$, fix $a,b\in\widehat{E}(L)/p$. Assume that $\{a,b\}_{L/L}\in\Fil^t(W)$ for some $t\geq 1$. Since by assumption $\res_{L/k}(N_{L/k}(\{a,b\}_{L/L}))\in\Fil^{t+1}(W)$, we may write $\res_{L/k}(N_{L/k}(\{a,b\}_{L/L}))$ in the form,
$\sum_{l=1}^N\{a_l,b_l\}_{L/L}$, for some $N\geq 1$ and some $a_l\in\mathcal{D}_L^n$, $b_l\in\mathcal{D}_L^m$ with $n+m\geq t+1$. We claim that $\{a,b\}_{L/k}=0$ if and only if $\sum_{l=1}^N\{a_l,b_l\}_{L/k}=0$. To see this note that the element $\sum_{l=1}^N\{a_l,b_l\}_{L/k}$ can be rewritten as, 
\[\sum_{l=1}^N\{a_l,b_l\}_{L/k}= N_{L/k}\left(\sum_{l=1}^N\{a_l,b_l\}_{L/L}\right)=N_{L/k}(\res_{L/k}(\{a,b\}_{L/k}))=(p^2-1)\{a,b\}_{L/k}.\] Since $p^2-1\equiv -1 \pmod p$, the claim follows. Using the  above argument allows us to reduce to proving that $\{a,b\}_{L/k}=0$, whenever $\{a,b\}_{L/L}\in\Fil^{t+1}(W)$. At this point we can continue inductively. The process is guaranteed to terminate in finitely many steps, since \autoref{filtration2} ($d$) implies that $\Fil^{2p^2}(W)=0$. 

From now we focus on proving the implication,
\[\{a,b\}_{L/L}\in\Fil^t(W)\Rightarrow\res_{L/k}(N_{L/k}(\{a,b\}_{L/L}))\in\Fil^{t+1}(W).\]
Let $G_0=\Gal(L/k)=\langle\sigma\rangle$ be the Galois group of $L/k$. Then for all $a,b\in\widehat{E}(L)/p$ we have,
 \[\res_{L/k}(N_{L/k}(\{a,b\}_{L/L}))=\sum_{r=0}^{p^2-2}\sigma^r(\{a,b\}_{L/L}).\]  Suppose that $a\in\mathcal{D}_L^i\setminus\mathcal{D}_L^{i+1}$, and $b\in\mathcal{D}_L^j\setminus\mathcal{D}_L^{j+1}$, for some $i,j\in\{1,\ldots,p^2\}$ with $i+j=t$. Note that if either $i$ or $j$ is equal to $p^2$, then \autoref{filtration2} and \autoref{resdec} imply that the extension $L_0=L\left(\frac{1}{p}b\right)$ is unramified over $L$. Using the surjectivity of the norm, $E_1(L_0)\xrightarrow{N_{L_0/L}} E_1(L)$, we immediately get that $\{a,b\}_{L/L}=0$. We may therefore reduce to the case when $i,j\in\{1,\ldots,p^2-1\}$. Consider the quotient $\Fil^t(W)/\Fil^{t+1}(W)$ and the well-defined map,
 \[\mathcal{D}_L^{i}/\mathcal{D}_L^{i+1}\otimes\mathcal{D}_L^{j}/\mathcal{D}_L^{j+1}\rightarrow\Fil^t(W)/\Fil^{t+1}(W).\]
Denote by 
$\overline{a}\otimes \overline{b}$ the image of $a\otimes b$ in $\mathcal{D}_L^{i}/\mathcal{D}_L^{i+1}\otimes\mathcal{D}_L^{j}/\mathcal{D}_L^{j+1}$. 
Moreover, let $\overline{\sigma}$ be the  endomorphism induced by $\sigma$,
\[\overline{\sigma}:\mathcal{D}_L^{i}/\mathcal{D}_L^{i+1}\otimes\mathcal{D}_L^{j}/\mathcal{D}_L^{j+1}\rightarrow\mathcal{D}_L^{i}/\mathcal{D}_L^{i+1}\otimes\mathcal{D}_L^{j}/\mathcal{D}_L^{j+1}.\] Fix a uniformizer $\pi_L$ of $L$ and set $\overline{u}=\sigma(\pi_L)/\pi_L\in\F_{p^2}^\times$. Then lemmas \ref{sigmaaction} and \ref{filtration3} (2) yield an equality,
\[\overline{\sigma}(\overline{a}\otimes\overline{b})=\overline{u}^i\overline{a}\otimes\overline{u}^{j}\overline{b}.\] 
Similarly to the proof of Claim 3 in \autoref{main2}, the tensor product $\mathcal{D}_L^{i}/\mathcal{D}_L^{i+1}\otimes\mathcal{D}_L^{j}/\mathcal{D}_L^{j+1}$ becomes a $\F_L$-vector space, and hence we can rewrite,
\[\overline{\sigma}(\overline{a}\otimes \overline{b})=\overline{u}^{i+j}\overline{a}\otimes \overline{b}.\] 
When $i+j$ is not divisible by $p^2-1$, we may proceed exactly as in the proof of Claim 3 in \autoref{main2} to deduce that 
\[\sum_{r=0}^{p^2-2}\overline{\sigma}^r(\overline{a}\otimes \overline{b})=\left(\sum_{r=0}^{p^2-2}(\overline{u}^{i+j})^r\right)\overline{a}\otimes \overline{b}=0,\] 
and hence $\res_{L/k}(N_{L/k}(\{a,b\}_{L/L}))\in\Fil^{t+1}(W)$. It remains to consider the cases when $i+j$ is a multiple of $p^2-1$. Since $2\leq i+j\leq 2p^2-2$, the only multiples of $p^2-1$ in that range are $p^2-1$ and $2p^2-2$. We consider each case separately. 

\smallskip
\noindent
\textit{Case 1:} Suppose that $i+j=2p^2-2$. This is only possible if $i=j=p^2-1$. In this case we can prove that $\{a,b\}_{L/L}=0$, imitating the proof of \autoref{ssing1}. 

\smallskip
\noindent
\textit{Case 2.1:} Suppose that $i+j=p^2-1$ and both $i,j$ are coprime to $p$. Using \eqref{decompose} we can decompose $a=(a_0,a_1)$, and $b=(b_0,b_1)$ with $a_0,b_0\in\overline{U}_L^{p^2+1}$, and $a_1,b_1\in\overline{U}_L^{p+1}$. Then \autoref{filtration2} gives us that $a_1\in\overline{U}_L^{p+i}\setminus\overline{U}_L^{p+i+1}$, $b_1\in\overline{U}_L^{p+j}\setminus\overline{U}_L^{p+j+1}$, while $a_0\in\overline{U}_L^{p^2+t}$, $b_0\in\overline{U}_L^{p^2+t'}$, for some $t,t'$ such that $pt>i$, and $pt'>j$. Notice that the symbols $\{(1,a_1),(b_0,1)\}_{L/k}$, $\{(a_0,1),(1,b_1)\}_{L/k}$ and $\{(a_0,1),(b_0,1)\}_{L/k}$ vanish due to previous considerations in the proof. It remains to show that $\{(1,a_1),(1,b_1)\}_{L/k}=0$. In fact, we claim that $\{(1,a_1),(1,b_1)\}_{L/L}=0$. We consider the finite extensions $L_1=L(\sqrt[p]{a_1})$, and $L_2=L(\sqrt[p]{b_1})$, both of which are totally ramified of degree $p$ over $L$. Using the usual vanishing trick (see \autoref{vanishingtrick}), it is enough to establish the following claim.

\smallskip
\noindent
\textit{Claim 2:} The assumption $i+j=p^2-1$ forces one of the elements to be a norm. That is, either $(1,a_1)\in N_{L_2/L}(E(L_2)/p)$, or $(1,b_1)\in N_{L_1/L}(E(L_1)/p)$. 

To prove Claim 2, we need to show that either $a_1$ lies in the image of the norm map $\overline{U}_{L_2}^{p^2}\xrightarrow{N_{L_2/L}}\overline{U}_L^{p}$, or $b_1$ lies in the image of $\overline{U}_{L_1}^{p^2}\xrightarrow{N_{L_1/L}}\overline{U}_L^{p}$. 
 Let $G_1=\Gal(L_1/L)$, $G_2=\Gal(L_2/L)$ be the Galois groups of $L_1, L_2$ respectively. Then the jump in the ramification filtration of $G_1$ (resp.\ of $G_2$) occurs at $s_1=p^2+p-(p+i)=p^2-i$ (resp.\ at $s_2=p^2-j$). Observe that since $i+j=p^2-1$, it follows that,
\[p+i=p^2+p-1-j>p^2-j\Rightarrow p+i>s_2.\] Since the computation is symmetric, we also get $p+j>s_1$. Then \autoref{psi} yields,
\begin{eqnarray*}\psi_{L_2/L}(p+i)=&&s_2+p(p+i-s_2)=p^2-j+p(p+i-p^2+j)\\=&&p^2-j+p(p-1)=2p^2-p-j.
\end{eqnarray*} A symmetric computation yields $\psi_{L_1/L}(p+j)=2p^2-p-i$. In order to prove Claim 2, we must show that at least one of
$\psi_{L_2/L}(p+i)$ and $\psi_{L_1/L}(p+j)$ is larger than $p^2$. Assume to contradiction that  $\psi_{L_1/L}(p+j)\leq p^2$, and $\psi_{L_2/L}(p+i)\leq p^2$. Adding these inequalities yields,
\[2p^2-p-j+2p^2-p-i\leq 2p^2\Rightarrow 2p^2-2p\leq i+j\Rightarrow 
 2p^2-2p\leq p^2-1\Rightarrow (p-1)^2\leq 0,
\] which is a contradiction. We conclude that either $\psi_{L_2/L}(p+i)>p^2$, or $\psi_{L_1/L}(p+j)>p^2$, which respectively give that $(1,a_1)\in N_{L_2/L}(E(L_2)/p)$, or $(1,b_1)\in N_{L_1/L}(E(L_1)/p)$. 
 
\smallskip
\noindent
\textit{Case 2.2:} Suppose that $i+j=p^2-1$ and one of the integers $i,j$ is divisible by $p$. Without loss of generality, assume that $i=pl$, for some $1\leq l<p$. We decompose $a=(a_0,a_1)$, $b=(b_0,b_1)$ with $a_0,b_0\in\overline{U}_L^{p^2+1}$, and $a_1,b_1\in\overline{U}_L^{p+1}$. \autoref{filtration3} and \cite[Lemma 2.1.4]{kawachi2002} yield that $a_0\in\overline{U}_L^{p^2+l}\setminus\overline{U}_L^{p^2+l+1}$, $a_1\in\overline{U}_L^{p+i+1}$, while $b_1\in\overline{U}_L^{p+j}\setminus\overline{U}_L^{p+j+1}$. Finally, $b_0\in\overline{U}_L^{p^2+t}$ for some $t $ such that $pt>j$. The symbols $\{(1,a_1), (b_0,1)\}_{L/k}$, $\{(1,a_1), (1, b_1)\}_{L/k}$ and $\{(a_0,1), (b_0,1)\}_{L/k}$  vanish due to earlier considerations in this proof. It remains to show that $\{(a_0,1), (1,b_1)\}_{L/k}=0$. We proceed similarly to Case 2.1, to show $\{(a_0,1), (1,b_1)\}_{L/L}=0$. Consider the finite extension $L_3=L(\sqrt[p]{a_0})$, which is totally ramified of degree $p$ over $L$. It is enough to establish the following claim.

\smallskip
\noindent
\textit{Claim 3:} The assumption $i+j=p^2-1$ forces  the element  $(1,b_1)\in E(L)/p$ to be a norm, namely $(1,b_1)\in N_{L_3/L}(E(L)/p)$. 

To prove Claim 3 we need to show that $b_1$ lies in the image of the norm $\overline{U}_{L_3}^{p^2}\xrightarrow{N_{L_3/L}}\overline{U}_L^{p}$. 
 Let $G_3=\Gal(L_3/L)$ be the Galois group of $L_3/L$. The jump in the ramification filtration of $G_3$  occurs at $s_3=p^2+p-(p^2+l)=p-l$. It is then clear that $p+j>s_3$, and hence we can compute $\psi_{L_3/L}(p+j)$ as follows,
 \[\psi_{L_3/L}(p+j)=s_3+p(p+j-s_3)=p-l+p(p+j-p+l)=p-l+p(j+l)=p-l+pj+pl.\] Using that $i+j=p^2-1$, and hence $pl+j=p^2-1$, the last equality can be rewritten as, \[\psi_{L_3/L}(p+j)=p-l+pj+p^2-1-j=p^2+p+(p-1)j-l-1.\] We want to prove that $\psi_{L_3/L}(p+j)>p^2 $. Equivalently, $(p-1)j+p-l>1$. But this is clear, since $l<p$ and $j\geq 1$. We conclude that $b_1\in\img(\overline{U}_{L_3}^{p^2}\xrightarrow{N_{L_3/L}}\overline{U}_L^{p})$, as required.

\end{proof}

\subsection{Extensions}
We close this section by giving some easy extensions of \autoref{main1}.

\begin{cor}\label{product} Let $k$ be a finite unramified extension of $\Q_p$ and $E_1,\ldots,E_r$ be elliptic curves over $k$ 
with good reduction.
 Assume that at most one of the curves has good supersingular reduction. 
  Let $X=E_1\times\cdots\times E_r$. Then, the Albanese kernel $F^2(X)$ is $p$-divisible. 
\end{cor}
\begin{proof}
Using the relation of $F^2(X)$ with the Somekawa $K$-groups given in   \eqref{F^2}, we are reduced to proving that the $K$-group $K(k;E_{i_1},\ldots, E_{i_{\nu}})$ is $p$-divisible, for every $2\leq\nu\leq r$ and $1\leq i_1<i_2<\cdots<i_{\nu} \leq r$. 
It is enough to show that the Mackey product $(E_{i_1}/p\otimes\cdots\otimes E_{i_\nu}/p)(k)=0$. 
Note that \autoref{main1} implies that this is true when $\nu=2$. The general case follows  by the fact that the product $\otimes$ in the category of Mackey functors is associative. 

\end{proof}

The following corollary follows by an easy descent argument. 

\begin{cor} \label{extensions} Let $k$ be a finite unramified extension of $\Q_p$. Let $X$ be a principal homogeneous space of an abelian variety $A$, such that  $X_L\simeq A_L$ for some finite extension $L/k$ of degree coprime to $p$. Suppose additionally that there is an isogeny $A\xrightarrow{\phi}E_1\times\cdots\times E_r$ of degree coprime to $p$, where $E_i$ are elliptic curves over $k$ satisfying the assumptions of \autoref{product}. Then, the groups $F^2(A)$ and $F^2(X)$ are $p$-divisible. 
\end{cor}

\begin{proof} We first show that the Albanese kernel $F^2(A)$ is $p$-divisible. Suppose that the isogeny $A\xrightarrow{\phi}E_1\times\cdots\times E_r$ has degree $n$. Let $\check{\phi}$ be the dual isogeny. These induce pushforward maps,
\[F^2(A)/p\xrightarrow{\phi_\star} F^2(E_1\times\cdots \times E_r)/p
\xrightarrow{\check{\phi}_\star}F^2(A)/p,\] with $\check{\phi}_\star\circ\phi_\star=n$. Since by assumption $n$ is coprime to $p$, multiplication by $n$ is injective on $F^2(A)/p$. At the same time $F^2(E_1\times\cdots \times E_r)/p=0$ by the previous corollary. We conclude that $F^2(A)$ is $p$-divisible. 

Next suppose that $X$ is a principal homogeneous space of $A$ such that $X_L\simeq A$ for some finite extension $L/k$ of degree $m$ which is coprime to $p$. Consider the projection $X_L\xrightarrow{g}X$, which induces a pushforward, $F^2(X_L)/p\xrightarrow{g_\star}F^2(X)/p$, and a pullback $F^2(X)/p\xrightarrow{g^\star} F^2(X_L)/p$, satisfying $g_\star\circ g^\star=m$. By the previous case we have $F^2(X_L)/p\simeq F^2(A)/p=0$, which forces $g^\star$ to vanish. Since $g^\star$ is also injective, the claim follows.

\end{proof}

\begin{prop}\label{split tori}
	Let $k$ be a finite unramified extension of $\Q_p$. 
	Let $A_1$ and $A_2$ be abelian varieties over $k$. 
	For each $i = 1,2$, we assume that 
	the connected component of the special fiber of the N\'eron model of $A_i$ 
	is a split torus. 
	Then,  $(A_1\otimes A_2)(k)$  is $p$-divisible. 
\end{prop}
\begin{proof}
	From the assumption on $A_i$, there exists a split torus 
	$T_i \simeq \G_m^{\oplus g_i}$ over $k$ with $g_i = \dim A_i$ 
	and a free abelian subgroup $L_i\subset T_i(k)$ such that 
	$T_i(k')/L_i \simeq  A_i(k')$  for any finite extension $k'/k$.
	This gives 
 surjections of Mackey functors $T_i/p \to A_i/p$, which induce a surjection
	\[
	 (T_1/p \otimes T_2/p)(k) \to  (A_1/p\otimes A_2/p)(k)\rightarrow 0
	\]
	(cf.\ \cite[Remark 4.2 (2)]{Yamazaki2005}).
	As the tori $T_1$ and $T_2$ split, we have 
	\[
	 (T_1/p \otimes T_2/p)(k)  \simeq (\G_m/p \otimes \G_m/p)(k)^{\oplus g_1g_2} \simeq (K_2(k)/p)^{\oplus g_1g_2} = 0, 
	\]
	where the last equality follows from  \cite[Chapter IX, Proposition 4.2]{Fesenko/Vostokov}.
	 This shows   that 
	$(A_1/p \otimes A_2/p)(k) =0$. 
		
\end{proof}
Recall that the Jacobian variety of a Mumford curve satisfies the assumption in the above proposition. Using the associativity of the Mackey products as in the proof of \autoref{product}, we obtain the following corollary.
\begin{cor}\label{Mumford curves}
	Let $k$ be a finite unramified extension of $\Q_p$. 
	Let $X = C_1\times \cdots \times C_r$ be a product of Mumford curves over $k$. 
	Then, the Albanese kernel 
	$F^2(X)$ is $p$-divisible.
\end{cor}

\smallskip
\section{Local-to-global results}\label{ltgsection} 
In this last section we focus on the local-to-global  \autoref{locatoglobalconj}, which constitutes one of the main motivations of this article. 
\begin{notn} For a $\Z$-module $M$ we will denote by $\widehat{M}:=\varprojlim_n M/n$ the completion of $M$. Let $F$ be a number field, that is, a finite extension of $\Q$. We will denote by $\Omega$, $\Omega_f$ and $\Omega_\infty$ the set of all places, all finite places and all infinite places of $F$ respectively. 
\end{notn}
\subsection{The Brauer group} Let $X$ be a smooth projective and geometrically connected variety over  $F$. 
The Brauer group  $\Br(X) = H^2_{\text{\'et}}(X,\mathbb{G}_m)$ of $X$ has a filtration,
\[\Br(X)\supset\Br_1(X)\supset\Br_0(X),\] induced by the Hochschild-Serre spectral sequence 
\[H^i(F,H^j_{\text{\'et}}(X_{\overline{F}},\G_m))\Rightarrow H^{i+j}_{\text{\'et}}(X,\G_m).\] The two subgroups are defined as follows; $ \Br_0(X):=\img(\Br(F)\rightarrow\Br(X))$, and  $\Br_1(X):=\ker(\Br(X)\rightarrow\Br(X_{\overline{F}}))$. The latter is usually referred in the literature as the \textit{algebraic Brauer group}, while the quotient $\Br(X)/\Br_1(X)$ is called the \textit{transcendental Brauer group}.  When $X$ is an abelian variety, a $K3$ surface (\cite[Theorem 1.1]{Skorobogatov/Zharin2008}), or a product of curves (\cite[Theorem B]{Skorobogatov/Zharin2014}) over a number field, the transcendental Brauer group is finite.

\subsection{The Brauer-Manin pairing} 
 For each place $v$ of $F$, put $X_v = X_{F_v}$.
Suppose $v$ is a finite place of $F$. There is a pairing, 
\[\langle\cdot,\cdot\rangle_v:CH_0(X_v)\times\Br(X_v)\rightarrow\Br(F_v)\simeq\Q/\Z,\] 
known as the Brauer-Manin pairing, which is defined using evaluation at closed points, and the invariant map of local class field theory, $\inv_v:\Br(F_v)\xrightarrow{\simeq}\Q/\Z$.  
Precisely, 
for each closed point $P\in X_v$ and $\alpha \in \Br(X_v)$, 
the pull-back of $\alpha$ along the closed point $P \to X_v$ 
is denoted by $\alpha(P) \in \Br(F_v(P))$, where $F_v(P)$ is the residue field of $P$. 
The paring above is defined by $\langle P,\alpha\rangle_v := \mathrm{Cor}_{F_v(P)/F_v}(\alpha(P))$.
In a similar manner, one can define a Brauer-Manin pairing for every real place $v$ of $F$, 
\[\langle\cdot,\cdot\rangle_v:CH_0(X_v)\times\Br(X_v)\rightarrow\Br(F_v)\simeq\Z/2\Z\hookrightarrow\Q/\Z.\] Note that in this case 
 the subgroup $\pi_{\overline{F}_v/F_v\star}(CH_0(X_v\otimes_{F_v}\overline{F}_v))$ of $CH_0(X_v)$ is contained in the left kernel of this pairing, where $X_{\overline{F}_v}\xrightarrow{\pi_{\overline{F}_v/F_v}}X$ is the projection. 

\begin{defnot}\label{adelicnotn} The \textit{adelic Chow group} of $X$ is defined as 
\[CH_{0,\mathbf{A}}(X)=\prod_{v\in\Omega}\overline{CH_0}(X_v),\] where for every finite place $v$ of $F$, we have an equality $\overline{CH_0}(X_v)=CH_0(X_v)$, while for every infinite place $v$, $\displaystyle\overline{CH_0}(X_v):=\frac{CH_0(X_v)}{\pi_{\overline{F}_v/F_v\star}(CH_0(X_v\otimes_{F_v}\overline{F}_v))}$.
 In a similar way, we define the adelic group $F^1_{\mathbf{A}}(X):=\prod_{v\in \Omega}\overline{F^1}(X_v)$ of zero-cycles of degree zero, and the adelic Albanese kernel, $F^2_{\mathbf{A}}(X):=\prod_{v\in \Omega}\overline{F^2}(X_v)$ (cf.\ Section~\ref{Kzero} for the definition of the filtration). 
\end{defnot}
Notice that for every infinite complex place $v$ we have an equality, $\overline{CH_0}(X_v)=0$. 
The local pairings induce a global pairing,
\[\langle\cdot,\cdot\rangle:CH_{0,\mathbf{A}}(X)\times\Br(X)\rightarrow\Q/\Z,\] defined by $\langle(z_v)_v,\alpha\rangle=\sum_v\langle z_v,\iota^\star(\alpha)\rangle_v$, where $\iota^\star$ is the pullback of $\iota: X_v\to X$.
We note that, if $v$ is a real place, then \cite[Th\'{e}or\`{e}me 1.3]{Colliot-Thelene1993} gives us that the group $\overline{F^1}(X_v)$ is isomorphic to a finite number of copies of $\Z/2\Z$. 

The short exact sequence of global class field theory, \[0\rightarrow\Br(F)\rightarrow\bigoplus_{v\in\Omega}\Br(F_v)\xrightarrow{\sum\inv_v}\Q/\Z\rightarrow 0,\] implies that the group $CH_0(X)$ lies in the left kernel of $\langle\cdot,\cdot\rangle$, thus giving rise to a complex, 
\begin{equation}\label{complex}CH_0(X)\stackrel{\Delta}{\longrightarrow}
CH_{0,\mathbf{A}}(X)\rightarrow\Hom(\Br(X),\Q/\Z).
\end{equation}

\subsection*{Behavior with respect to filtrations}

Under the Brauer-Manin pairing, the filtration $CH_0(X)\supset F^1(X)\supset F^2(X)\supset 0$ is compatible with the filtration $\Br(X)\supset\Br_1(X)\supset\Br_0(X)$ of the Brauer group. Namely, 
restricting the map $CH_{0,\mathbf{A}}(X)\to\Hom(\Br(X),\Q/\Z)$ to $F^1_{\mathbf{A}}(X)$, it factors through \[F^1_{\mathbf{A}}(X)\rightarrow\Hom(\Br(X)/\Br_0(X),\Q/\Z)\hookrightarrow\Hom(\Br(X),\Q/\Z).\] Similarly, restricting the latter to $F^2_{\mathbf{A}}(X)$, it factors through, 
\[F^2_{\mathbf{A}}(X)\rightarrow\Hom\left(\frac{\Br(X)}{\Br_1(X)},\Q/\Z\right)\hookrightarrow\Hom(\Br(X),\Q/\Z).\] 
This compatibility follows by considering the local pairings, $CH_0(X_v)\rightarrow\Hom(\Br(X_v),\Q/\Z)$ and using local Tate duality. For a proof of this compatibility we refer to \cite[Proof of Proposition 3.1]{Yamazaki2005} and \cite[Proof of Theorem 6.9]{Gazaki2015}.

\subsection{The conjecture} We are interested in the following conjecture. 
\begin{conj}[{\cite[Section 4]{Colliot-Thelene/Sansuc1981}, \cite[Section 7]{Kato/Saito1986}, \cite[Conjecture~1.5 (c)]{Colliot-Thelene1993} and \cite[Conjecture\ ($E_0$)]{Wittenberg2012}}]\label{bigconj} 
Let $X$ be a smooth projective geometrically connected variety over a  number field $F$. 
 The following complex is exact, 
\begin{equation}\label{complex2}
\widehat{F^1(X)}\xrightarrow{\Delta}
\widehat{F^1_{\mathbf{A}}(X)}\rightarrow\Hom(\Br(X)/\Br_0(X),\Q/\Z),\end{equation}  where $F^1_{\mathbf{A}}(X)$ is the group defined in \ref{adelicnotn}. 
\end{conj}
When $E$ is an elliptic curve, it is a theorem of Cassels (\cite{Cassels1964}) that \autoref{bigconj} is true, if the Tate-Shafarevich group of $E$ contains no nonzero divisible element; in particular, it is true if it is finite. This result has been generalized by Colliot-Th\'{e}l\`{e}ne \cite[paragraphe 3]{CT1997} to all curves, assuming that the Tate-Shafarevich group of their Jacobian contains no nonzero divisible element.  We next consider what happens for a product $X=C_1\times C_2$ of two curves over $F$. 
\begin{prop}\label{complexreduction} Let $X=C_1\times C_2$ be a product of smooth projective curves over a number field $F$. Let $\Alb_X=J_1\times J_2$ be the Albanese variety of $X$, where $J_1, J_2$ are the Jacobian varieties of $C_1, C_2$ respectively. Assume that $X$ contains a $k$-rational point and that the Tate-Shafarevich group of $\Alb_X$ contains no nonzero divisible element. Then the exactness of \eqref{complex2} can be reduced to the exactness of the following complex, 
\begin{equation}\label{complex1}\widehat{F^2(X)}\stackrel{\Delta}{\longrightarrow}
\widehat{F^2_{\mathbf{A}}(X)}\rightarrow\Hom\left(\frac{\Br(X)}{\Br_1(X)},\Q/\Z\right).\end{equation}
\end{prop}

\begin{proof} It follows by \cite[Corollary 2.4.1]{Raskind/Spiess2000} that we have a direct sum  decomposition,
\[F^1(X)\simeq J_1(F)\oplus J_2(F)\oplus F^2(X),\] and the same holds over $F_v$, for every place $v$ of $F$. This decomposition passes to the completions. Thus, proving exactness of \eqref{complex2} amounts to proving exactness of \eqref{complex1} and exactness of the following complex,
\begin{equation}\label{complex3}
\widehat{\Alb_X(F)}\xrightarrow{\Delta} \prod_{v\in\Omega}\widehat{\Alb_{X_v}}(F_v)\rightarrow \Hom(\Br_1(X)/\Br_0(X),\Q/\Z).
\end{equation} It follows directly by the Hochshchild-Serre spectral sequence, 
\[E_2^{i,j}=H^i(F, H^j(X_{\overline{F}},\G_m))\Rightarrow H^{i+j}(X,\G_m)\] that we have a surjection $H^1(F,\Pic(X_{\overline{F}}))\rightarrow \Br_1(X)/\Br_0(X)\rightarrow 0.$ 
Namely, $H^1(F,\Pic(X_{\overline{F}}))=E_2^{1,1}$, while the quotient $\Br_1(X)/\Br_0(X)$ is precisely the graded piece $gr_1(E^2)=gr_1(\Br(X))=E_{\infty}^{1,1}$. Note that the differential $d_2^{1,1}:E_2^{1,1}\rightarrow E_2^{3,0}$ vanishes, since $E_2^{3,0}=H^3(F,\overline{F}^\times)=0$. This is a byproduct of global class field theory (cf.\ \cite[Remark 6.7.10]{Poonen2017}). 
Thus, we get a surjection $E_2^{1,1}\rightarrow E_\infty^{1,1}\rightarrow 0$. This induces an injection on the dual groups, $0\rightarrow\Hom(\Br_1(X)/\Br_0(X),\Q/\Z)\rightarrow\Hom(H^1(F,\Pic(X_{\overline{F}})),\Q/\Z)$. Moreover, there is a map $\Hom(H^1(F,\Pic(X_{\overline{F}})),\Q/\Z)\rightarrow\Hom(H^1(F,\Pic^0(X_{\overline{F}})),\Q/\Z)$ induced by the short exact sequence of $G_F$-modules,
\[0\rightarrow\Pic^0(X_{\overline{F}})\rightarrow\Pic(X_{\overline{F}})\rightarrow NS(X_{\overline{F}})\rightarrow 0.\]
The variety $\Pic^0(X)$ is the dual abelian variety to $\Alb_X$. Since the latter is equal to the product $J_1\times J_2$, it is self dual. We conclude that the exactness of \eqref{complex3} follows by the exactness of the following complex
\[\widehat{\Alb_X(F)}\xrightarrow{\Delta} \prod_{v\in\Omega}\widehat{\Alb_{X_v}(F_v)}\xrightarrow{\beta} \Hom(H^1(F, \Alb_X),\Q/\Z).\] 
Here the map $\beta:\prod_{v\in\Omega}\widehat{\Alb_{X_v}(F_v)}\rightarrow \Hom(H^1(F, \Alb_X),\Q/\Z)$ is obtained by all the local isomorphisms $\Alb_{X_v}(F_v)\simeq \Hom(H^1(F_v,\Alb_{X_v}),\Q/\Z)$ induced by local Tate duality (cf.\ \cite[I. Corollary 3.4]{Milne2006}) and composing it with the map \[\prod_{v\in\Omega}\Hom(H^1(F_v,\Alb_{X_v}),\Q/\Z)\xrightarrow{\sum_v} \Hom(H^1(F,\Alb_{X}),\Q/\Z).\] The fact that the map $\sum_v$ is well-defined follows by \cite[I. Lemma 6.3]{Milne2006}. 
This complex is known to be exact under the assumption on the Tate-Shafarevich group of $\Alb_X$ (cf.\ \cite[II. 5.6(b)]{Milne2006}).

\end{proof}


\subsection{Elliptic curves with potentially good reduction}\label{CM} 
In this section we consider a product $X=E_1\times E_2$ of two elliptic curves over  $F$.  We assume that for $i=1,2$ the elliptic curve $E_i$ has potentially good reduction at all finite places of $F$. 
It is known that 
an elliptic curve over $F$ has potentially good reduction 
if and only if  its $j$-invariant is integral 
(\cite[Chapter~VII, Proposition 5.5]{Silverman2009}). An important class of elliptic curves with this property are elliptic curves with  
complex multiplication 
(\cite[p.\ 225]{Deuring1941} and \cite[Theorem 7]{Serre/Tate1968}, see also \cite[Chapter~VII, Exercise 7.10]{Silverman2009}). 

Suppose that $v_1,\ldots, v_r$ are all the places of bad reduction of $X$. 
Then, there exists some finite extension $L_{v_i}$ of $F_{v_i}$ such that $X_{L_{v_i}}$ has good reduction. We set $n_i := [L_{v_i}:F_{v_i}]$.


 \begin{lem}\label{additivered} 
Let $i\in\{1,\ldots, r\}$. Let $p$ be a prime number such that $v_i \nmid p$ and $p$ is coprime to $n_i$. 
 Then the Albanese kernel  $F^2(X_{v_i})$ is $p$-divisible.
 \end{lem} 
 
 \begin{proof}
Consider the projection $X_{L_{v_i}}\xrightarrow{\pi_{L_{v_i}/F_{v_i}}} X_{v_i}$, and let $CH_0(X_{L_{v_i}})\xrightarrow{\pi_{L_{v_i}/F_{v_i}\star}} CH_0(X_{v_i})$, and $CH_0(X_{v_i})\xrightarrow{\pi_{L_{v_i}/F_{v_i}}^{\star}} CH_0(X_{L_{v_i}})$ be the induced push-forward and pull back maps respectively. Since $X_{L_{v_i}}$ has good reduction and $v_i \nmid p$, it follows by \cite[Corollary 0.10]{Saito/Sato2010} that the Albanese kernel $F^2(X_{L_{v_i}})$ is $p$-divisible. Moreover, the endomorphism \[\frac{F^2(X_{v_i})}{p}\stackrel{n_i}{\longrightarrow} \frac{F^2(X_{v_i})}{p}\] can be factored as $n_i=\pi_{L_{v_i}/F_{v_i}\star}\circ\pi_{L_{v_i}/F_{v_i}}^{\star}$. Since we assumed that $p$ does not divide $n_i$, 
 this forces the multiplication by $n_i$ to be injective modulo $p$. At the same time the map $\pi_{L_{v_i}/F_{v_i}\star}$ is the zero map. We conclude that the group $F^2(X_{v_i})/p$ vanishes. 
 
 \end{proof}

\begin{defn}
 Let $S$ be the set of rational primes consisting of: 
\begin{itemize}
\item all the primes $p$ such that $v_i\mid p$ for some $1\le i\le r$,
\item all the prime divisors of $\prod_{i=1}^r n_i$,
\item all the primes $p$ such that both $E_1, E_2$ have good supersingular reduction at $v$ for some place $v\in \Omega_f$ that lies above $p$, 
\item all the ramified primes in the extension $F/\Q$, and
\item $p = 2$.
\end{itemize} 
Moreover, let $T_S$ be the set of all integers whose prime divisors do not belong to $S$, namely  
\[T_S=\{n\geq 1: p \nmid n, \text{ for all }p\in S\}.\]  
\end{defn}

 The following corollary follows directly from \autoref{main2}.
\begin{cor}\label{localtoglobal} Let $X=E_1\times E_2$ be the product of two elliptic curves over $F$. Assume that for $i=1,2$ the elliptic curve $E_i$ has potentially good reduction at all finite places of $F$. 
Then, we have 
\[
\prod_{v\in\Omega}\varprojlim_{n\in T_S}\frac{\overline{F^2}(X_v)}{n}=0.
\] 
In particular, the following complex is exact,
\[\varprojlim_{n\in T_S}\frac{F^2(X)}{n}\xrightarrow{\Delta}\prod_{v\in\Omega}\varprojlim_{n\in T_S}\frac{\overline{F^2}(X_v)}{n}\xrightarrow{\varepsilon}\Hom\left(\frac{\Br(X)}{\Br_1(X)},\Q/\Z\right).\]
\end{cor}
\begin{proof} 
Fix $v\in\Omega$ and $n\in T_S$. 
We will show that the group $\overline{F^2}(X_{v})$ is $n$-divisible. 
It is enough to show that $\overline{F^2}(X_v)$ is $p$-divisible,  
for each prime divisor $p$ of $n$. 
By definition of the set $T_S$, 
we have $p\not\in S$. 

\smallskip
\noindent
\textit{Case 1:} Suppose $v\in\Omega_\infty$ is a real place of $F$. 
By our assumption, $p$ is odd. Since $\overline{F^2}(X_v)=(\Z/2\Z)^s$, for some $s\geq 0$ (\cite[Th\'eor\`em 1.3]{Colliot-Thelene1993}), we conclude that $\overline{F^2}(X_v)$ is $p$-divisible. 

\smallskip
\noindent
\textit{Case 2:} Suppose $v \in\Omega_f$ and $v \nmid p$. 
If $X_v$ has good reduction, then $F^2(X_v)$ is $p$-divisible (\cite[Corollary 0.10]{Saito/Sato2010}). 
If $X_v$ has bad reduction, then 
it follows by \autoref{additivered} that  $F^2(X_v)$ is $p$-divisible. 

\smallskip
\noindent
\textit{Case 3:} Suppose $v\in\Omega_f$ and $v \mid p$. 
By $p\not\in S$, the surface $X_v$ has good reduction. 
Moreover, the extension $F_v/\Q_p$ is unramified, and at least one of the curves $E_{1v}$, $E_{2v}$ over $F_v$ has good ordinary reduction. It then follows by \autoref{main1} that the group $F^2(X_v)$ is $p$-divisible.  

\end{proof}

\begin{exmp} Consider the product $X=E_1\times E_2$ of the elliptic curves given by the Weierstrass equations $y^2=x^3+x$ and $y^2=x^3+1$ respectively. The curve $E_1\otimes_\Q \overline{\Q}$ has complex multiplication by $\Z[i]$, while $E_2\otimes_\Q \overline{\Q}$ has complex multiplication by $\Z[\omega]$, where $\omega$ is a primitive third of unity. Then \autoref{localtoglobal} in this case reads as follows: for every prime $p\geq 5$ such that $p\not\equiv 11$ mod $12$, the group $\varprojlim_n F^2_{\mathbf{A}}(X)/p^n=0$. 
\end{exmp}

\smallskip
%

\begin{thebibliography}{CTSSDa}

\bibitem[Be84]{Beilinson1984}
A.~A.~Be\u\i{}linson, \emph{Higher regulators and values of {$L$}-functions},
  Current problems in mathematics, {V}ol. 24, Itogi Nauki i Tekhniki, Akad.
  Nauk SSSR, Vsesoyuz. Inst. Nauchn. i Tekhn. Inform., Moscow, 1984,
  pp.~181--238. \MR{760999}

\bibitem[Blo84]{Bloch1984}
S.~Bloch, \emph{Algebraic cycles and values of {$L$}-functions}, J. Reine
  Angew. Math. \textbf{350} (1984), 94--108. \MR{743535}

\bibitem[Cas64]{Cassels1964}
J.~W.~S. Cassels, \emph{Arithmetic on curves of genus {$1$}. {VII}. {T}he dual
  exact sequence}, J. Reine Angew. Math. \textbf{216} (1964), 150--158.
  \MR{169849}
  


 \bibitem[CT95]{Colliot-Thelene1993}
J.-L. Colliot-Th{\'e}l{\`e}ne, \emph{L'arithm\'etique du groupe de {C}how
  des z\'ero-cycles}, J. Th\'eor. Nombres Bordeaux \textbf{7} (1995), no.~1,
  51--73, Les Dix-huiti{\`e}mes Journ{\'e}es Arithm{\'e}tiques (Bordeaux,
  1993). \MR{1413566 (97i:14006)}


\bibitem[CT99]{CT1997}
\bysame, \emph{Conjectures de type local-global sur
  l'image des groupes de {C}how dans la cohomologie \'{e}tale}, Algebraic
  {$K$}-theory ({S}eattle, {WA}, 1997), Proc. Sympos. Pure Math., vol.~67,
  Amer. Math. Soc., Providence, RI, 1999, pp.~1--12. \MR{1743234}

\bibitem[CT05]{Colliot-Thelene2005}
\bysame, \emph{Un th\'{e}or\`eme de finitude pour le groupe de {C}how des
  z\'{e}ro-cycles d'un groupe alg\'{e}brique lin\'{e}aire sur un corps
  {$p$}-adique}, Invent. Math. \textbf{159} (2005), no.~3, 589--606.
  \MR{2125734}

   \bibitem[CTR91]{Colliot-TheleneRaskind1991} 
   J.-L. Colliot-Th{\'e}l{\`e}ne, W. Raskind, \emph{Groupe de Chow de codimension deux de vari\'{e}t\'{e}s d\'{e}finies sur un corps de nombres: un th\'{e}or\`{e}me de finitude pour la torsion}, Invent. Math. \textbf{105} (1991), p. 221-245. \MR{1115542}
  
  \bibitem[CTS81]{Colliot-Thelene/Sansuc1981}
J.-L. Colliot-Th\'{e}l\`ene and J.-J.~Sansuc, \emph{On the {C}how
  groups of certain rational surfaces: a sequel to a paper of {S}. {B}loch},
  Duke Math. J. \textbf{48} (1981), no.~2, 421--447. \MR{620258}
  
 \bibitem[CTSSDa]{CT/Sansuc/SwinnertonDyerI}
J.-L. Colliot-Th\'{e}l\`ene, J.-J.~Sansuc and P.~Swinnerton-Dyer, \emph{Intersections of two quadrics and {C}h\^{a}telet surfaces. {I}}, J.\ Reine Angew.\ Math \textbf{373} (1987), 
  37--107. \MR{870307}

 \bibitem[CTSSDb]{CT/Sansuc/SwinnertonDyerII}
\bysame, \emph{Intersections of two quadrics and {C}h\^{a}telet surfaces. {II}}, J. Reine Angew. Math \textbf{374} (1987),
  72--168. \MR{876222}  
 
  
  
  \bibitem[Deur41]{Deuring1941} 
  M.~Deuring, \emph{Die Typen der Multiplicatorenringe elliptischer  Funktionenk\"{o}rper}. Abh. Math. Sem. Hamburg, \textbf{14} (1941), pp. 197--272. 

\bibitem[FV02]{Fesenko/Vostokov}
I.~B.~Fesenko and S.~V.~Vostokov, \emph{Local fields and their extensions},
  second ed., Translations of Mathematical Monographs, vol. 121, American
  Mathematical Society, Providence, RI, 2002, With a foreword by I. R.
  Shafarevich. \MR{1915966}

\bibitem[Gaz15]{Gazaki2015}
E.~Gazaki, \emph{On a filtration of {$CH_0$} for an abelian variety},
  Compos. Math. \textbf{151} (2015), no.~3, 435--460. \MR{3320568}

\bibitem[Gaz19]{Gazaki2019}
\bysame, \emph{Some results about zero-cycles on abelian and semi-abelian varieties}, Math. Nachrichten, \textbf{292} (2019), pp. 1716--1726. \MR{3994298}

\bibitem[GL18]{Gazaki/Leal2018}
E.~{Gazaki} and I.~{Leal}, \emph{{Zero-cycles on a product of
  elliptic curves over a $p$-adic field}}, arXiv:1802.03823 [math.NT]. To appear in IMRN. 
  
\bibitem[Haz74]{Hazewinkel1974}
M.~Hazewinkel, \emph{On norm maps for one dimensional formal groups. {I}.
  {T}he cyclotomic {$\Gamma$}-extension}, J. Algebra \textbf{32} (1974),
  89--108. \MR{349692}

\bibitem[Hir16]{Hiranouchi2014}
T.~Hiranouchi, \emph{Milnor {$K$}-groups attached to elliptic curves over a
  {$p$}-adic field}, Funct. Approx. Comment. Math. \textbf{54} (2016), no.~1,
  39--55. \MR{3477733}

\bibitem[HH13]{Hiranouchi/Hirayama2013}
T.~Hiranouchi and S.~Hirayama, \emph{On the cycle map for products of
  elliptic curves over a {$p$}-adic field}, Acta Arith. \textbf{157} (2013),
  no.~2, 101--118. \MR{3007290}




\bibitem[{Ier}18]{Ieronymou2018}
E.~{Ieronymou}, \emph{{The Brauer-Manin obstruction for zero-cycles on K3
  surfaces}}, arXiv:1804.03908 [math.AG].

\bibitem[Kah92a]{Kahn92a}
B.~Kahn, \emph{The decomposable part of motivic cohomology and bijectivity of
  the norm residue homomorphism}, Algebraic {$K$}-theory, commutative algebra,
  and algebraic geometry ({S}anta {M}argherita {L}igure, 1989), Contemp. Math.,
  vol. 126, Amer. Math. Soc., Providence, RI, 1992, pp.~79--87.

\bibitem[Kah92b]{Kahn92b}
\bysame, \emph{Nullit\'e de certains groupes attach\'es aux vari\'et\'es
  semi-ab\'eliennes sur un corps fini; application}, C. R. Acad. Sci. Paris
  S\'er. I Math. \textbf{314} (1992), no.~13, 1039--1042.
  
\bibitem[KY13]{Kahn/Yamazaki2013}
B.~Kahn and T.~Yamazaki, \emph{Voevodsky's motives and Weil reciprocity}, Duke Math J. \textbf{162} (2013), no.~14, pp. 2751--2796. \MR{3127813}

  
\bibitem[KS83]{Kato/Saito1983}
K.~Kato and S.~Saito, \emph{Unramified class field theory of arithmetic surfaces}, Annals of Math. \textbf{118} (1983) no.~2, pp.~241--275. \MR{717824}

\bibitem[KS86]{Kato/Saito1986}
\bysame,  \emph{Global class field theory of arithmetic
  schemes}, Applications of algebraic {$K$}-theory to algebraic geometry and
  number theory, {P}art {I}, {II} ({B}oulder, {C}olo., 1983), Contemp. Math.,
  vol.~55, Amer. Math. Soc., Providence, RI, 1986, pp.~255--331. \MR{862639}

\bibitem[Kat73]{Katz1973}
N.~M.~Katz, \emph{{$p$}-adic properties of modular schemes and modular
  forms}, 69--190. Lecture Notes in Mathematics, Vol. 350. \MR{0447119}

\bibitem[Kaw02]{kawachi2002}
M.~Kawachi, \emph{Isogenies of degree $p$ of elliptic curves over local
  fields and kummer theory}, Tokyo J. of Math. \textbf{25} (2002), no.~2,
  247--259.


\bibitem[KS16]{Kerz/Saito2016}
M.~Kerz and S.~Saito, \emph{Chow group of 0-cycles with modulus and
  higher-dimensional class field theory}, Duke Math. J. \textbf{165} (2016),
  no.~15, 2811--2897. \MR{3557274}

\bibitem[KS03]{Kollar/Szabo2003}
J.~Koll\'{a}r and E.~Szab\'{o}, \emph{Rationally connected varieties
  over finite fields}, Duke Math. J. \textbf{120} (2003), no.~2, 251--267.
  \MR{2019976}


\bibitem[LR99]{LangerRaskind}
A.~Langer and W.~Raskind, \emph{0-cycles on the self-product of a CM  elliptic curve over $\Q$}, J. Reine Angew. Math. \textbf{516} (1999), 1--26. \MR{1724615}

\bibitem[LS96]{LangerSaito}
A.~Langer and S.~Saito, \emph{Torsion zero-cycles on the self-product of a modular elliptic curve}, Duke Math J. \textbf{85} (1996), no.~2, 315--357. \MR{1417619}


\bibitem[Maz72]{Mazur1972}
B.~Mazur, \emph{Rational points of abelian varieties with values in towers
  of number fields}, Invent. Math. \textbf{18} (1972), 183--266. \MR{0444670}

\bibitem[MS82]{Merkurev/Suslin1982}
A.~S.~Merkur{\cprime}ev and A.~A. Suslin, \emph{{$K$}-cohomology of
  {S}everi-{B}rauer varieties and the norm residue homomorphism}, Izv. Akad.
  Nauk SSSR Ser. Mat. \textbf{46} (1982), no.~5, 1011--1046, 1135--1136.
  
\bibitem[Mil06]{Milne2006}
J.~S.~Milne, \emph{Arithmetic duality theorems}, second ed., BookSurge, LLC,
  Charleston, SC, 2006. \MR{2261462}

\bibitem[Mum68]{Mumford1968}
D.~Mumford, \emph{Rational equivalence of {$0$}-cycles on surfaces}, J. Math.
  Kyoto Univ. \textbf{9} (1968), 195--204. \MR{0249428 (40 \#2673)}

\bibitem[PS95]{Parimala/Suresh1995}
R.~Parimala and V.~Suresh, \emph{Zero-cycles on quadric fibrations: finiteness
  theorems and the cycle map}, Invent. Math. \textbf{122} (1995), no.~1,
  83--117. \MR{1354955}

\bibitem[Poo17]{Poonen2017} 
B.~Poonen, \emph{Rational points on varieties}, Graduate Studies in Mathematics,  volume 186 (2017), American Math Society, Providence, RI. \MR{3729254}

\bibitem[RS00]{Raskind/Spiess2000}
W.~Raskind and M.~Spiess, \emph{Milnor {$K$}-groups and zero-cycles on
  products of curves over p-adic fields}, Compositio Math. \textbf{121} (2000),
  no.~1, 1--33. \MR{1753108 (2002b:14007)}

\bibitem[Roit80]{Roitman1980}
A.~A.~Roitmann, \emph{The torsion of the group of {$0$}-cycles modulo rational equivalence}, Annals of Math. Second Series \textbf{111} (1980), no.~3, 553--569. \MR{577137}

\bibitem[SS10]{Saito/Sato2010}
S.~Saito and K.~Sato, \emph{A finiteness theorem for zero-cycles over
  {$p$}-adic fields}, Ann. of Math. (2) \textbf{172} (2010), no.~3, 1593--1639,
  With an appendix by Uwe Jannsen. \MR{2726095} 
  
  \bibitem[Sal93]{Salberger}
  P.~ Salberger, \emph{Chow groups of codimension two and l-adic realizations of motivic cohomology}, S\'{e}minaire de th\'{e}orie des nombres, Paris 1991/1992, \'{e}d Sinnou David, Progress in Mathematics \textbf{116} (1993), p. 247-277.  

\bibitem[Ser72]{Serre1972}
J.-P.~Serre, \emph{Propri\'{e}t\'{e}s galoisiennes des points d'ordre
  fini des courbes elliptiques}, Invent. Math. \textbf{15} (1972), no.~4,
  259--331. \MR{387283}

\bibitem[Ser79]{serre1979local}
\bysame, \emph{Local fields}, Graduate Texts in Mathematics, vol.~67,
  Springer-Verlag, New York-Berlin, 1979, Translated from the French by Marvin
  Jay Greenberg. \MR{554237}
  
 \bibitem[ST68]{Serre/Tate1968}
 J.-P.~Serre and J.~Tate, \emph{Good reduction of abelian varieties}, Annals of Math. \textbf{88} (1968), pp. 492--517. \MR{236190}

  

\bibitem[Sil09]{Silverman2009}
J.~H.~Silverman, \emph{The arithmetic of elliptic curves}, second ed.,
  Graduate Texts in Mathematics, vol. 106, Springer, Dordrecht, 2009.
  \MR{2514094}
  
  \bibitem[Sko01]{Skorobogatov2001} A.~N.~Skorobogatov, \emph{Torsors and rational points}, volume\textbf{144} of Cambridge Tracts in  Mathematics, Cambridge University Press, Cambridge 2001. 

\bibitem[SZ08]{Skorobogatov/Zharin2008}
A.~N.~Skorobogatov and Y.~G.~Zarhin, \emph{A finiteness theorem for the
  {B}rauer group of abelian varieties and {$K3$} surfaces}, J. Algebraic Geom.
  \textbf{17} (2008), no.~3, 481--502. \MR{2395136}

\bibitem[SZ14]{Skorobogatov/Zharin2014}
\bysame, \emph{The {B}rauer group and the {B}rauer-{M}anin set of products of
  varieties}, J. Eur. Math. Soc. (JEMS) \textbf{16} (2014), no.~4, 749--768.
  \MR{3191975}
  
  
\bibitem[Som90]{Somekawa1990}
M.~Somekawa, \emph{On {M}ilnor {$K$}-groups attached to semi-abelian
  varieties}, $K$-Theory \textbf{4} (1990), no.~2, 105--119. \MR{1081654
  (91k:11052)}


\bibitem[St09]{Stix2009}
J.~Stix, \emph{A course on finite flat group schemes and {$p$}-divisible groups}, course lecture notes, \url{https://www.uni-frankfurt.de/52288632/Stix_finflat_Grpschemes.pdf}. 


\bibitem[Tat76]{Tate1976}
J.~Tate, \emph{Relations between {$K_{2}$} and {G}alois cohomology}, Invent.
  Math. \textbf{36} (1976), 257--274. \MR{0429837}

\bibitem[Wit12]{Wittenberg2012}
O.~Wittenberg, \emph{Z\'{e}ro-cycles sur les fibrations au-dessus d'une
  courbe de genre quelconque}, Duke Math. J. \textbf{161} (2012), no.~11,
  2113--2166. \MR{2957699}
  
\bibitem[Wit18]{Wittenberg2018}
\bysame, \emph{Rational points and zero-cycles on rationally connected varieties over number fields}, Algebraic geometry: {S}alt {L}ake {C}ity 2015, Part 2, p. 597-635, Proceedings of Symposia in Pure Mathematics 97, American Mathematical Society, Providence, RI, 2018. \MR{3821185}


\bibitem[Yam05]{Yamazaki2005}
T.~Yamazaki, \emph{On {C}how and {B}rauer groups of a product of {M}umford
  curves}, Math. Ann. \textbf{333} (2005), no.~3, 549--567. \MR{2198799}

\end{thebibliography}

\def\cprime{$'$}
\providecommand{\bysame}{\leavevmode\hbox to3em{\hrulefill}\thinspace}
\providecommand{\MR}{\relax\ifhmode\unskip\space\fi MR }
\providecommand{\MRhref}[2]{%
  \href{http://www.ams.org/mathscinet-getitem?mr=#1}{#2}
}
\providecommand{\href}[2]{#2}

\end{document}